\documentclass[11pt]{article}

\usepackage{fourier} 
\usepackage{soul} 
\usepackage{enumerate} 

\usepackage{amsfonts,amssymb,amsmath,amsthm}
\usepackage{bbm}

\usepackage{xcolor}
\DeclareMathOperator*{\argmax}{argmax}

\numberwithin{equation}{section}

\def\11{\mathbbm{1}}
\def\E{\mathbb{E}}
\def\P{\mathbb{P}}
\def\R{\mathbb{R}}

\def\N{\mathbb{N}}

\def\d{\partial}
\def\Z{\mathbb{Z}}

\def\cM{{\cal M}}
\def\cP{{\cal P}}
\def\cW{{\cal W}}
\def\cL{{\cal L}}

\def\bard{d}

\newtheorem{thm}{Theorem}[section]
\newtheorem{lem}[thm]{Lemma}
\newtheorem{cor}[thm]{Corollary}

\newtheorem{prop}[thm]{Proposition}

\theoremstyle{remark}
\newtheorem{rem}{Remark}
\newtheorem{exa}{Example}

\begin{document}

\title{Lower bound for the coarse Ricci curvature of continuous-time pure jump processes}

\author{Denis Villemonais$^{1,2}$}

\footnotetext[1]{Universit\'e de Lorraine, IECN, Campus Scientifique, B.P. 70239,
  Vand{\oe}uvre-l\`es-Nancy Cedex, F-54506, France} \footnotetext[2]{Inria, TOSCA team,
  Villers-l\`es-Nancy, F-54600, France.\\
  E-mail: Denis.Villemonais@univ-lorraine.fr}

\maketitle

\begin{abstract}
  We obtain a lower bound for the coarse Ricci curvature of continuous
  time pure jump Markov processes, with an emphasis on interacting
  particle systems. Applications to several models are provided, with
  a detailed study of the herd behavior of a simple model of
  interacting agents.
\end{abstract}

\section{Introduction}
\label{sec:intro}

Let $(E,d)$ be a Polish space. Fix $N\geq 1$ and consider a continuous
time pure jump particle system of $N$ particles
$(\bar{X}_t)_{t\geq 0}=(X^1_t,\ldots,X^N_t)_{t\geq 0}$ evolving in
$E^N$.  We assume that the process is non-explosive and that its
infinitesimal generator $\cL$ is given, for all
$\bar{x}=(x_1,\ldots,x_N)\in E^N$ and any bounded measurable function
$f:E^N\rightarrow\R$, by
\begin{align*}
  \cL{\color{black} f(\bar x)}=\sum_{i=1}^N \int_E \left(f(x_1,\ldots,x_{i-1},y,x_{i+1},\ldots,x_N)-f(x_1,\ldots,x_n)\right)\,F_i\left(x_i,\bar{x},dy\right),
\end{align*}
where the terms $F_i(x_i,\bar{x},\cdot)$ are finite non-negative
measures on $E$, measurable with respect to $x_i$ and $\bar{x}$ and
such that, for some (and hence for all) $\bar{x}\in E^N$,
$\int d(x_i,y)F_i(x_i,\bar{x},dy)<\infty$.  Our main result, stated in
Section~\ref{sec:main-result}, provides a lower bound for the coarse
Ricci curvature of $\bar{X}$ evolving in $E^N$ endowed with the metric
\begin{align*}
  \bard(\bar{x},\bar{y})=\frac{1}{N}\sum_{i=1}^N d(x_i,y_i),\ \forall \bar{x}=(x_1,\ldots,x_N),\ \bar{y}=(y_1,\ldots,y_N)\in E^N.
\end{align*} 
We recall that the coarse Ricci curvature of the conti\-nuous-time
Markov process $\bar{X}$ is the largest constant $\sigma$ satisfying,
for all $t\geq 0$,
\begin{align*}
  \cW_\bard\left(\P(\bar{X}_t\in\cdot\mid \bar{X}_0=\bar{x}),\P(\bar{X}_t\in\cdot\mid \bar{X}_0=\bar{y})\right)\leq e^{-\sigma t}\,\bard(\bar{x},\bar{y}),\ \forall \bar{x},\bar{y}\in E^N,
\end{align*}
where $\cW_\bard$ denotes the Wasserstein distance (see
Section~\ref{sec:main-result} for references and details).  A lower
bound on $\sigma$ provides a measure of the instantaneous convergence
rate to a unique stationary distribution (see for
instance~\cite{Chen1994}). It also entails spectral gap inequalities
and concentration inequalities
(see~\cite{Ollivier2009,JoulinOllivier2010,Joulin2007,
  Joulin2009,Veysseire2012,Veysseire2012a,DjelloutGuillinEtAl2004}).

\medskip In Section~\ref{sec:N-equal-1}, as a first application of our
main result, we provide a general lower bound for the coarse Ricci
curvature of simple ($N=1$) continuous time pure jump processes. The
time-continuous version of the coarse Ricci curvature has often been
considered not practical because of the lack of general and practical
lower bounds, see~\cite{Ollivier2009} and~\cite{Chen2010}, contrarily
to the discrete time case. The computation of our lower bound mainly
requires the computation of a Wasserstein distance between measures,
similarly to the discrete time case. We refer the reader
to~\cite{AlfonsiCorbettaEtAl2016} for a different approach
based on Kantorovich potentials. In Example~\ref{exa:tv-N-equal1}, we
consider the case where $d$ is the {\color{black}trivial} distance and where the
jump measures admit a density with respect to a common non-negative
measure on $E$.  In Example~\ref{exa:bd-proc-1d}, we check that the
lower bound provided by our main result in the case of birth and death
processes is in fact equal to the coarse Ricci curvature, as computed
explicitly in~\cite{Chen1994,Joulin2009}. This entails that, at least
in some simple cases, this lower bound is sharp. We also show in
Example~\ref{exa:bd-proc-1d-modified} how to compute non trivial lower
bounds for the coarse Ricci curvature of a modified birth and death
process, using our result and a slight extension of Vallender's
Theorem~\cite{Vallender1973} for the computation of the Wasserstein
distance between probability measures on the real line (see
Lemma~\ref{lem:vallender}).

\medskip In Section~\ref{sec:inter-agents}, we study a simple model of
interacting agents whose individual behavior is influenced in a
non-linear way by the behavior of the other agents: each agent wanders
randomly in a complete graph and also changes its position to a new
one, depending on a function of the number of agents in this
position. This dynamic is modeled by a system of $N$ particles
evolving in the complete finite graph $E$ of size $\# E\geq 2$: we
assume that there exist $T>0$ and a function $f:[0,1]\rightarrow\R_+$
such that any agent jumps from state $x$ to $y\in E$ with the
following rate
\begin{align}
\label{eq:intro-alpha}
x\rightarrow y\text{ with rate }\frac{T}{\# E}+f\left(\frac{\text{Number of agents in y}}{N}\right).
\end{align}
In this model, $T$ is the temperature of the system and $f$ is a
preference function. For instance, with an increasing function $f$
with high convexity, the agents will give higher preferences to
positions that are already favored by many other agents; with a larger
temperature $T$, the agents act more independently. Our aim is to
determine characteristics of $f$ and values of $T$ for which a herd
behavior occurs or not in this model. By a herd behavior, we mean a
meta-stable state of the whole particle system where a majority of the
agents share the same position for a long time. Note that this model
can be written in the settings of the present paper, by setting, for
all $x,y\in E$ and $\bar{x}\in E^N$,
\begin{align*}
F_i(x,\bar{x},\{y\})=\frac{T}{\# E}+f\left(\frac{\sum_{i=1}^N \11_{x_i=y}}{N}\right),\ \forall y\in E.
\end{align*}
The existence of the phase without herd behavior is obtained using the
results of Section~\ref{sec:main-result}, while the existence of the
phase with herd behavior is proved using large deviation results
obtained in~\cite{DupuisEllisEtAl1991,DupuisRamananEtAl2016}.

In Section~\ref{sec:other-models}, lower bounds of the coarse Ricci
curvature for several models are obtained: we consider zero range
dynamics in Subsection~\ref{sec:zero-range-dyn}, Fleming-Viot type
systems and some natural extensions in
Subsection~\ref{sec:Fleming-Viot}, birth and death processes in
mean-field type interaction in Subsection~\ref{sec:BD-mean-field} and
system of particles whose jump measures admit a density with respect
to the Lebesgue measure or the counting measure in
Subsection~\ref{sec:with-density}.

\section{Definitions and main result}
\label{sec:main-result}

\subsection{Definitions and reminders about the Wasserstein distance}
\label{sec:wass-remainder}
Fix $N\geq 1$ and consider the Polish space $(E^N,\bard)$. Let
$\cP_d(E^N)$ (respectively $\cM_d(E^N)$) denote the set of
probability measures (respectively of non-negative finite measures)
$\mu$ on $E^N$ such that, for some (and hence for all) $x\in E^N$,
$\int d(x,y)\mu(dy)<\infty$. The Wasserstein distance $\cW_d$ between
two probability measures $\mu$ and $\nu$ on $E^N$ belonging to
$\cP_d(E^N)$ is defined as
\begin{align}
\label{eq:wass-def}
\cW_d(\mu,\nu)=\inf_\pi \int_{E^N\times E^N} d(x,y)\,\pi(dx,dy),
\end{align}
where the infimum is taken over all probability measures $\pi$ on
$E^N\times E^N$ such that $\pi(\cdot,E^N)=\mu(\cdot)$ and
$\pi(E^N,\cdot)=\nu(\cdot)$ ($\pi$ is called a coupling measure for
$\mu$ and $\nu$). It is well known that the infimum in the above
definition is attained and the state space
$\left(\cP_d(E^N),\cW_d\right)$ is a complete state space (see for
instance Lemma~5.2 and Theorem~5.4 in~\cite{Chen2004}).  The
Wasserstein distance is also referred to as the Kantorovich metric
(which one may consider a more suitable name given the historical
precedence~\cite{Vershik2013}) and is a particular instance of the
Kantorovich-Rubistein norm (with $d$ replaced by a suitable cost
function and $\pi$ taken in the set of measures such that
$\pi(\cdot,E^N)-\pi(E^N,\cdot)=\mu(\cdot)-\nu(\cdot)$,
see~\cite[Chapter 6]{RachevKlebanovEtAl2013} for relations between the
different types of norms).

The Wasserstein distance can also be easily extended to positive measures with the same mass: for all $\alpha>0$ and any probability measures $\mu,\nu$ on $E^N$, we set
\begin{align}
\label{eq:Wd-alpha}
\cW_d(\alpha \mu,\alpha\nu)=\inf_\pi \int_{E^N\times E^N} d(x,y)\,\pi(dx,dy)=\alpha \cW_d(\mu,\nu),
\end{align}
where the infimum is  taken over all measures $\pi$ on $E^N\times E^N$ with mass $\alpha$ and such that $\pi(\cdot,E^N)=\mu(\cdot)$ and $\pi(E^N,\cdot)=\nu(\cdot)$.
Note that if a coupling $\pi$ realizes the minimum in the definition of $\cW_d(\mu,\nu)$, then $\alpha\pi$ realizes the minimum in the definition of $\cW_d(\alpha \mu,\alpha\nu)$. Such couplings are also referred to as optimal couplings.

Given a continuous time Markov process $(\bar{X}_t)_{t\geq 0}$ evolving in $E^N$, the coarse Ricci curvature of $\bar{X}$ (as coined by Ollivier~\cite{Ollivier2009}, see also~\cite{Joulin2007} and \cite[Remark 2.3]{Joulin2009} where this quantity is called \textit{the Wasserstein curvature})  is the largest constant $\sigma\in[-\infty,+\infty]$ satisfying, for all $t\geq 0$,
\begin{align*}
\cW_d\left(\P(\bar{X}_t\in\cdot\mid \bar{X}_0=\bar{x}),\P(\bar{X}_t\in\cdot\mid \bar{X}_0=\bar{y})\right)\leq e^{-\sigma t}d(\bar{x},\bar{y}),\ \forall \bar{x},\bar{y}\in E^N.
\end{align*}
In the discrete time setting,  we refer the reader to~\cite[Theorem~2.1 and Lemma~2.1]{Zhang1999} for a first use of this concept in a general setting and to~\cite{Ollivier2009} for a systematic study.
If $\sigma$ is positive, then the completeness of $\left(\cP_d(E^N),\cW_d\right)$ implies that the process admits a unique stationary distribution $\mu_\infty$, that $\mu_\infty\in \cP_d(E^N)$ and that, for all $t\geq 0$ and any initial distribution $\mu\in\cP_d(E^N)$,
\begin{align*}
\cW_d\left(\P_\mu(\bar{X}_t\in\cdot),\mu_\infty\right)\leq e^{-\sigma t}\cW_d(\mu,\mu_\infty).
\end{align*}
Note that this concept is closely related to the optimal coupling theory developed by Chen (see for instance~\cite{Chen1994,Chen2005}). Several implications of this notion have been proved in~\cite{Joulin2007,Joulin2009}, where Joulin obtains Poisson type deviation inequalities for jump type processes. We also refer the reader to~\cite[Section~3.2]{ChafaiJoulin2013} for a link between coarse Ricci curvature and functional inequalities. For general state space processes and for diffusion processes, we refer the reader to the works of Veysseire, where a systematic study of the coarse Ricci curvature has been conducted (see~\cite{Veysseire2012,Veysseire2012a}) with nice implications on concentration inequalities and spectral gap estimates. Let us also mention that estimates on the coarse Ricci curvature  of a continuous time process immediately provides estimates for the curvature of its discrete time included Markov chain, which also implies several interesting properties (see the works of Ollivier~\cite{Ollivier2009,Ollivier2010} and references therein).

Estimates on the coarse Ricci curvature can be obtained using the coupling of Markov processes. Let $\cL$ be the infinitesimal generator of $\bar{X}$. We recall (see \cite[Definition~5.12]{Chen2004}) that a coupling operator $\cL^c$ of $\cL$ is an operator acting on functions $f:E^N\times E^N:\rightarrow\R$ and such that
\begin{align*}
\cL^c f(\bar{x},\bar{y})&=\cL g(\bar{x})\quad\text{if }f(\bar{x},\bar{y})=g(\bar{x})\quad\forall \bar{x},\bar{y}\in E^N\\
\cL^c f(\bar{x},\bar{y})&=\cL g(\bar{y})\quad\text{if }f(\bar{x},\bar{y})=g(\bar{y})\quad\forall \bar{x},\bar{y}\in E^N
\end{align*}
for some function $g:E^N\rightarrow\R$. Since $\cL$ is the infinitesimal generator of a pure jump non-explosive process (see \cite[Chapter~2]{Chen2005}), any coupling operator $\cL^c$ is also non-explosive and $\cL^c d(\bar{x},\bar{y})$ is well defined for all $\bar{x},\bar{y}\in E^N$. A common way to prove that the coarse Ricci curvature of a pure jump Markov process is bounded {\color{black}from} below by a constant $c\in\R$ is to prove that there exists a coupling operator $\cL^c$ of $\cL$ such that
\begin{align*}
\cL^c d(\bar{x},\bar{y})\leq -c d(\bar{x},\bar{y}),\ \forall \bar{x},\bar{y}\in E^N.
\end{align*}
Indeed, standard localization arguments and Dynkin's formula entail that, for a Markov process $(\bar{Y},\bar{Z})$ with generator $\cL^c$ satisfying the above inequality,
\begin{align*}
\E(d(\bar{Y}_t,\bar{Z}_t) \mid \bar{Y}_0=\bar{y},\bar{Z}_0=\bar{z})\leq e^{-ct}d(\bar{y},\bar{z}),\ \forall \bar{y},\bar{z}\in E^N.
\end{align*}
Now, since the law of $\P((\bar{Y}_t,\bar{Z}_t)\in\cdot\mid \bar{Y}_0=\bar{y},\bar{Z}_0=\bar{z})$ is a coupling measure for $\P(\bar{X}_t\in\cdot\mid \bar{X}_0=\bar{y})$ and $\P(\bar{X}_t\in \cdot\mid \bar{X}_0=\bar{z})$, we deduce that, for all $\bar{y},\bar{z}\in E^N$,
\begin{align*}
\cW_d(\P(\bar{X}_t\in\cdot\mid \bar{X}_0=\bar{y}),\P(\bar{X}_t\in\cdot\mid \bar{X}_0=\bar{z}))\leq \E(d(\bar{Y}_t,\bar{Z}_t) \mid \bar{Y}_0=\bar{y},\bar{Z}_0=\bar{z})\leq e^{-ct}d(\bar{y},\bar{z}),
\end{align*}
and hence that the coarse Ricci curvature of $\bar{X}$ is bounded {\color{black} from} below by $c$.

\begin{rem}
\label{rem:coupling}
The above strategy also applies to Markov processes that are not of
pure jump types \textcolor{black}{and to cost functions $d$ that are
  not distance functions.} 
For diffusion processes, we refer the reader to~\cite{ChenLi1989} and
to~\cite[Corollary~1.4]{RenesseSturm2005} for necessary and sufficient
conditions in the case where the drift derives from a potential. We
also refer the reader to~\cite{Eberle2011,Eberle2016} with an
introduction to parallel coupling and the construction of \textit{ad
  hoc} distances on the state space. Computation of the coarse Ricci
curvature for diffusion processes on manifold has also been studied
by~Veysseire~\cite{Veysseire2012a}. For piecewise deterministic
processes, we refer the reader to~\cite[Lemma~5.2]{CloezHairer2015}
and~\cite[Theorem~2.3]{ChafaiMalrieuEtAl2010}. Original coupling
approaches are also provided
in~\cite{Majka2016,Majka2015,CattiauxGuillin2014}.
\end{rem}

\subsection{Main result}

We introduce the family of functions $(J_d^{x,y})_{x,y\in E}$ from $\cM_d(E)^2$ to $\R$, defined for all $m_1,m_2\in\cM_d(E)$ by
\begin{align*}
J_{d}^{x,y}(m_1,m_2)&=
 \cW_d(m_1+m_2(E)\delta_x, m_2+m_1(E)\delta_y)-(m_1(E)+m_2(E))d(x,y),
\end{align*}
\textcolor{black}{where $\delta_x$ denotes the Dirac measure at point $x$ and $m_2(E)\delta_x$ is the product of the scalar $m_2(E)$ by $\delta_x$.}
Note that the finite measures $m_1$ and $m_2$ can have different masses. Properties of $J_d^{x,y}$ are provided in Subsection~\ref{sec:Jdxy} and explicit computations of lower bounds for $J_d^{x,y}$ are provided in the subsequent sections.

The following theorem is the main result of this paper.  The particular case $N=1$ is detailed in Subsection~\ref{sec:N-equal-1} and applications to particle systems are provided in Sections~\ref{sec:inter-agents} and~\ref{sec:other-models}.

\medskip

\begin{thm}
\label{thm:main}
Consider the Markov process $\bar{X}$ with generator $\cL$ given in the introduction. 
Then there exists a coupling operator $\cL^c$ of $\cL$ such that, for all $\bar{x},\bar{y}\in E^N$,
\begin{align*}
\cL^c d(\bar{x},\bar{y})= \frac{1}{N}\sum_{i=1}^N J_{d}^{x_i,y_i}(F_i(x_i,\bar{x},\cdot),F_i(y_i,\bar{y},\cdot)).
\end{align*}
In particular, the coarse Ricci curvature $\sigma$ of the process $(\bar{X}_t)_{t\geq 0}$ satisfies
$$
\sigma\geq -\sup_{\bar{x},\bar{y}\in E^N}\frac{\frac{1}{N}\sum_{i=1}^N J_{d}^{x_i,y_i}(F_i(x_i,\bar{x},\cdot),F_i(y_i,\bar{y},\cdot))}{d(\bar{x},\bar{y})}.
$$
\end{thm}

\begin{rem}
  This result remains valid under a more general setting. For instance, if $E$ is the subset of a Polish space $(F,\rho)$ and if $d:E\times E\rightarrow \R_+$ is a continuous non-negative function, then the infimum in the definition of the Wasserstein distance is  attained~\cite[Theorem~4.1]{Villani2009} and there exists a measurable selection of such optimal couplings~\cite[Corollary~5.22]{Villani2009}, so that the proof of Theorem~\ref{thm:main} holds true. An other important setting, which will be used in the following sections, is the case where $E$ is a separable metric space endowed with its Borel $\sigma$-field and $d$ is the {\color{black} trivial} distance (\textit{i.e.} $d(x,y)=\11_{x\neq y}$ for all $x,y\in E$). In this case, $\cW_d$ is one half of the total variation distance{\color{black}, that is}
    \begin{align*}
\color{black}      \cW_d(\mu_1,\mu_2)=\frac{1}{2}\,\left\|\mu_1-\mu_2\right\|_{TV}=\sup_{A\subset E} \mu_1(A)-\mu_2(A),
    \end{align*}
 and the optimal coupling in the definition of $\cW_d$ is a measurable function of the Jordan Hahn decomposition of signed measures (which is itself measurable because of the regularity of Borel probability measures on metric spaces~\cite[Theorem~1.1]{Billingsley1999} and because of the separability assumption), so that the proof of Theorem~\ref{thm:main} still applies.
\end{rem}

\textcolor{black}{
  \begin{rem}
    In Theorem~\ref{thm:main}, we obtain, using coupling methods,
    lower bounds on the coarse Ricci curvature of a system of
    particles from the behavior of individual particles. This idea of
    reconstituting transport distance bounds on Markov chains on
    product spaces from the behavior of marginals via suitable
    couplings was already used by Talagrand and Marton, see for
    instance~\cite{Marton1996,Marton2015} and references therein.
  \end{rem}
}

\begin{proof}[Proof of Theorem~\ref{thm:main}]
Fix $\bar{x}=(x_1,\ldots,x_N)\in E^N$ and $\bar{y}=(y_1,\ldots,y_N)\in E^N$. We define the operator
\begin{align*}
\cL^c f(\bar{x},\bar{y})&=\sum_{i=1}^N \int_{E\times E}  \left[f\left(\bar{x}+(u-x_i)e_i,\bar{y}+(v-y_i)e_i\right)-f(\bar{x},\bar{y})\right]n^i_{x_i,y_i,\bar{x},\bar{y}}(du,dv),
\end{align*}
where $e_i$ is the $i^{th}$ element of the canonical base of $\{0,1\}^N$ and where $n^i_{x_i,y_i,\bar{x},\bar{y}}$ is a coupling measure between the positive measures $F_i(x_i,\bar{x},\cdot)+F_i(y_i,\bar{y},E)\delta_{x_i}$ and $F_i(y_i,\bar{y},\cdot)+F_i(x_i,\bar{x},E)\delta_{y_i}$ such that 
\begin{multline}
\label{eq:def-ni}
\cW_{d}(F_i(x_i,\bar{x},\cdot)+F_i(y_i,\bar{y},E)\delta_{x_i},F_i(y_i,\bar{y},\cdot)+F_i(x_i,\bar{x},E)\delta_{y_i})\\=
 \int_{E\times E} d(u,v) \,n^i_{x_i,y_i,\bar{x},\bar{y}}(du,dv).
 \end{multline}
 Note that $n^i$ can be constructed as a measurable function of $(x_i,y_i,\bar{x},\bar{y})$ by \cite[Theorem~1.1]{Zhang1999}, so that $\cL^c$ is the infinitesimal generator of a pure jump process.
 
Let us first check that $\cL^c$ is a coupling operator for $\cL$. We have, for any bounded measurable function $f:E^N\times E^N\rightarrow\R$ such that $f(\bar{x},\bar{y})=g(\bar{x})$ for some function $g:E^N\rightarrow\R$ (so that $f$ only depends on $\bar{x}$),
\begin{align*}
\cL^c f(\bar{x},\bar{y})&=\sum_{i=1}^N \int_{E\times E}  \left[g\left(\bar{x}+(u-x_i)e_i\right)-g(\bar{x})\right]n^i_{x_i,y_i,\bar{x},\bar{y}}(du,dv).
\end{align*}
for all $\bar{x},\bar{y}\in E^N$.
For each couple $x_i,y_i\in E$, we observe that the integral with respect to the first marginal of $n^i_{x_i,y_i,\bar{x},\bar{y}}$ is equal to the integral with respect to $F_i(x_i,\bar{x},\cdot)+F_i(y_i,\bar{y},E)\delta_{x_i}$. Hence, since the integral of $g\left(\bar{x}+(u-x_i)e_i\right)-g(\bar{x})$ with respect to $\delta_{x_i}$ is $0$, we obtain that
\begin{align*}
\cL^c f(\bar{x},\bar{y})&= \sum_{i=1}^N \int_{E\times E}  \left[g\left(\bar{x}+(u-x_i)e_i\right)-g(\bar{x})\right]\,F_i(x_i,\bar{x},du)=\cL g(\bar{x}).
\end{align*}
By symmetry of the roles of $\bar{x}$ and $\bar{y}$, we deduce that $\cL^c$ is indeed a coupling operator for $\cL$.

\bigskip

Now our aim is to prove that, for all $\bar{x},\bar{y}\in E^N$,
$$
\cL^c d(\bar{x},\bar{y})\leq \frac{1}{N}\sum_{i=1}^N J_{d}^{x_i,y_i}(F_i(x_i,\bar{x},\cdot),F_i(y_i,\bar{y},\cdot)),
$$
which will conclude the proof of Theorem~\ref{thm:main}. 
We have
\begin{align*}
\cL^c d(\bar{x},\bar{y})&=\sum_{i=1}^N \int_{E\times E}  \frac{d(u,v)-d(x_i,y_i)}{N}\,n^i_{x_i,y_i,\bar{x},\bar{y}}(du,dv),
\end{align*}
where, by definition of $n^i_{x_i,y_i,\bar{x},\bar{y}}$,
\begin{align*}
\int_{E\times E} d(x_i,y_i)\,n^i_{x_i,y_i,\bar{x},\bar{y}}(du,dv)&=d(x_i,y_i)\,\,n^i_{x_i,y_i,\bar{x},\bar{y}}(E\times E)\\
&=d(x_i,y_i)\left(F_i(x_i,\bar{x},E)+F_i(y_i,\bar{y},E)\right),
\end{align*}
and hence, using equality~\eqref{eq:def-ni},
\begin{align*}
\cL^c d(\bar{x},\bar{y})&=\sum_{i=1}^N\frac{\,J_d^{x,y}(F_i(x_i,\bar{x},\cdot),F_i(y_i,\bar{y},\cdot))}{N}.
\end{align*}
\end{proof}

\begin{rem}
\label{rem:perturbed-diff}
One can use the results of this section to study Markov processes obtained from other types of infinitesimal generators. For instance, let $L$ be the infinitesimal generator of $N$ independent diffusion processes or piecewise deterministic processes and consider the infinitesimal generator $L+\cL$, which can be seen as a perturbation of independent random paths (given by $L$) by jumps with dependence (given by $\cL$). If there exists a coupling $L^c$ of $L$ such that $L^c d\leq -cd$ for some constant $c$ (see Remark~\ref{rem:coupling}), then one can expect to prove that
\begin{align*}
\sigma\geq c-\sup_{\bar{x},\bar{y}\in E^N}\frac{\frac{1}{N}\sum_{i=1}^N J_{d}^{x_i,y_i}(F_i(x_i,\bar{x},\cdot),F_i(y_i,\bar{y},\cdot))}{d(\bar{x},\bar{y})},
\end{align*} 
using the coupling $L^c+\cL^c$ of $L+\cL$.  Two difficulties arise :
first, one needs to ensure that this coupling operator defines a
proper Markov process; second, that it is possible to apply this
coupling operator to $d$. Since it is more intricate to check these
properties for general Markov processes, we mainly restrict our
attention to the case of pure jump type infinitesimal
generators. However, the method used here and in the particular
examples of the next sections can be adapted to these situations, as
in the following example.
\end{rem}

\begin{exa}
\label{exa:perturbed-diff}
Consider a process evolving in $\R^N$ with generator
\begin{align*}
{\cal H} f(\bar{x})=\sum_{i=1}^N \left( \frac{1}{2}\frac{\d^2 f}{\d x_i^2}(x_i)-b x_i \frac{\d f}{\d x_i}(x_i)\right)+\cL f(\bar{x}),
\end{align*}
where $b>0$ is a constant and $\cL$ is the pure jump type
infinitesimal generator of the introduction.  This is the generator of
a system of $N$ particles evolving as independent Ornstein Uhlenbeck
processes between their jumps (several properties of similar processes
with jumps are investigated in~\cite{Wang2011}). The jumps occur with
respect to a jump measure which depends on the position of the whole
system. Now, consider the following coupling generator
\begin{align*}
{\cal H}^c f(\bar{x},\bar{y})=\sum_{i=1}^N H^c_i f(\bar{x},\bar{y}) + \cL^c f(\bar{x},\bar{y}),
\end{align*}
where $H_i^c$ is the basic coupling (also called the parallel coupling) for Ornstein Uhlenbeck processes (see for instance~\cite[Example~2.5]{ChenLi1989}), which satisfies, for all $x_i,y_i\in\R^2$, $H^c_i d(x_i,y_i)\leq -b \,d(x_i,y_i)$ and where $\cL^c$ is the coupling for $\cL$ obtained from Theorem~\ref{thm:main}. If $\sigma$ is the lower bound provided by
 Theorem~\ref{thm:main} for the coarse Ricci curvature of the pure jump part, then
\begin{align*}
{\cal H}^c d(\bar{x},\bar{y}) \leq  -(b+\sigma)d(\bar{x},\bar{y}),
\end{align*}
so that the coarse Ricci curvature of the process generated by ${\cal H}$ is bounded {\color{black} from} below by $b+\sigma$. 
\end{exa}

\subsection{Some properties of $J_d^{x,y}$}
\label{sec:Jdxy}
One of the difficulties of the continuous time setting is that the
jump measures do not, in general, share the same mass, contrarily to
the discrete time case, where one can use the standard Wasserstein
distance to compare transition probabilities~\cite{Ollivier2009}. In
the definition of $J_d^{x,y}$, the quantity
$\cW_d(m_1+m_2(E)\delta_x,m_2+m_1(E)\delta_y)$ is used to compare
measures with different masses. However, this is clearly not a proper
distance between non-negative measures since this quantity is equal to
zero for all couple $m_1,m_2$ such that $m_1\in \R_+.\delta_y$ and
$m_2\in\R_+.\delta_x$\textcolor{black}{, where $\R_+.\delta_y$ is
  defined as the set of non-negative measures
  $\{\alpha\delta_y,\ \alpha\in\R_+\}$}. Proper generalizations of the
Wasserstein distance exist in the literature (such as the flat
metric~\cite{Dudley2002} and the generalized $W^{1,1}_1$ Wasserstein
distance~\cite{PiccoliRossi2016}\textcolor{black}{, see also the recent
  developments in~\cite{Chizat2018,Kondratyev2016,Liero2018} with
  applications to convergence of measure valued dynamical systems}),
but are not directly relevant in our context.

Fix $x,y\in E$. The aim of this section is to provide some properties
of $J_d^{x,y}$, which will be useful to derive upper bounds and hence
to apply Theorem~\ref{thm:main}.

\begin{prop}
\label{prop:homogeneity}
For all $m_1,n_1,m_2,n_2\in\cM_d(E)$ and all $\alpha>0$, we have
\begin{align}
\label{eq:J-prop2}
J_d^{x,y}(\alpha m_1,\alpha m_2)=\alpha J_d^{x,y}(m_1,m_2)
\end{align}
and
\begin{align}
\label{eq:J-prop1}
J_{d}^{x,y}(m_1+n_1,m_2+ n_2)&\leq J_d^{x,y}(m_1,m_2)+J_d^{x,y}(n_1,n_2).
\end{align}
\end{prop}

\begin{proof}
Equality~\eqref{eq:J-prop2} is an immediate consequence of the definition of $J_d^{x,y}$ and of~\eqref{eq:Wd-alpha}.

Let $\pi_1$ and $\pi_2$ be two coupling measures realizing the minimum in the definition of $\cW_d(m_1+m_2(E)\delta_x,m_2+m_1(E)\delta_y)$ and of $\cW_d(n_1+n_2(E)\delta_x,n_2+n_1(E)\delta_y)$ respectively. Then $\pi_1+\pi_2$ is a coupling measure for $m_1+n_1+(m_2(E)+n_2(E))\delta_x$ and $m_2+n_2+(m_1(E)+n_1(E))\delta_y$, so that
\begin{align*}
&\cW_d(m_1+n_1+(m_2(E)+n_2(E))\delta_x,m_2+n_2+(m_1(E)+n_1(E))\delta_y)\\
&\phantom{\cW_d(m_1} \leq \int_{E\times E} d(u,v)\,(\pi_1+\pi_2)(du,dv)\\
&\phantom{\cW_d(m_1} = \cW_d(m_1+m_2(E)\delta_x,m_2+m_1(E)\delta_y)+\cW_d(n_1+n_2(E)\delta_x,n_2+n_1(E)\delta_y).
\end{align*}
Subtracting $(m_1(E)+n_1(E)+m_2(E)+n_2(E))d(x,y)$ leads to~\eqref{eq:J-prop1}.
\end{proof}

The following inequality is in general a crude estimate, but it is in some cases useful and sharp (as in Example~\ref{exa:bd-proc-1d}).
\begin{prop}
\label{prop:classical-coupling}
We have, for all $m_1,m_2\in\cM_d(E)$,
\begin{align*}
J_d^{x,y}(m_1,m_2)\leq \int_E [d(u,y)-d(x,y)]\,m_1(du)+\int_E [d(x,v)-d(x,y)]\,m_2(dv).
\end{align*}
\end{prop}

\begin{proof}
Since $m_1\otimes \delta_y+\delta_x\otimes m_2$ is a coupling measure for $m_1+m_2(E)\delta_x$ and $m_2+m_1(E)\delta_y$, we have 
\begin{align*}
\cW_d(m_1+m_2(E)\delta_x,m_2+m_1(E)\delta_y)&\leq \int_{E\times E} d(u,v) \,(m_1\otimes \delta_y+\delta_x\otimes m_2)(du,dv)\\
&=\int_E d(u,y)\,m_1(du)+\int_E d(x,v)\,m_2(dv).
\end{align*}
Subtracting $(m_1(E)+m_2(E))d(x,y)$, one obtains the desired inequality.
\end{proof}

The following property implies in particular that, if $m_1$ and $m_2$ are two probability measures, then $J_d^{x,y}(m_1,m_2)$ is smaller than $\cW_d(m_1,m_2)-d(x,y)$. It also implies that, for measures $m_1$ and $m_2$ on $E$ such that $m_1(E)\geq m_2(E)$, then 
\begin{align}
\label{eq:alt-bound}
J_{d}^{x,y}(m_1,m_2)&\leq
 \cW_d(m_1, m_2+(m_1(E)-m_2(E))\delta_y)-m_1(E)d(x,y).
\end{align}
\begin{prop}
We have, for all $m_1,m_2\in\cM_d(E)$,
\begin{align*}
J_d^{x,y}(m_1,m_2)=\min_{a,b} \cW_d(m_1+a\delta_x,m_2+b\delta_y)-(m_1(E)+a)\,d(x,y),
\end{align*}
where $a,b$ are taken in the set of real numbers such that $m_1+a\delta_x$ and $m_2+b\delta_y$ are non-negative measures on $E$ with equal mass, \textit{i.e.} such that $m_1({\color{black} E})+a\geq 0$, $m_2({\color{black} E})+b\geq 0$ and $m_1(E)+a=m_2(E)+b$. In addition, the minimum is attained for all $a\geq m_2(E)$ (or equivalently $b\geq m_1(E)$).
\end{prop}

\begin{proof}
Taking $a=m_2(E)$ and $b=m_1(E)$, one deduces that $J_d^{x,y}(m_1,m_2)$ is larger than the right hand side. 

Let us now prove the converse inequality. Let $a$ and $b$ be two real numbers such that $m_1+a\delta_x$ and $m_2+b\delta_y$ are non-negative measures on $E$ with equal mass and denote by $\pi$ a coupling which realizes the minimum in the definition of $\cW_d(m_1+a\delta_x,m_2+b\delta_y)$.

If $a< m_2(E)$ (and hence $b<m_1(E)$), then $\pi+(m_2(E)-a)\delta_x\otimes \delta_y$ is a coupling measure for $m_1+m_2(E)\delta_x$ and $m_2+m_1(E)\delta_y$, so that
\begin{align*}
\cW_d(m_1+m_2(E)\delta_x,m_2+m_1(E)\delta_y)&\leq \int_{E\times E} d(u,v)\,(\pi+(m_2(E)-a)\delta_x\otimes \delta_y)(du,dv)\\
&=\cW_d((m_1+a\delta_x,m_2+b\delta_y)+(m_2(E)-a)d(x,y).
\end{align*}
Subtracting $(m_1(E)+m_2(E))d(x,y)$ implies that
\begin{align*}
J_d^{x,y}(m_1,m_2)\leq \cW_d(m_1+a\delta_x,m_2+b\delta_y)-(m_1(E)+a)\,d(x,y).
\end{align*}
If $a>m_2(E)$ (and hence $b>m_1(E)$), then 
\begin{align*}
m_1(\{x\})+a&=\pi(\{x\},E)=\pi(\{x\},\{y\})+\pi(\{x\},E\setminus\{y\})\\
&\leq \pi(\{x\},\{y\})+\pi(E,E\setminus\{y\})=\pi(\{x\},\{y\})+m_2(E\setminus\{y\}).
\end{align*}
We deduce that $\pi(\{x\},\{y\})\geq m_1(\{x\})+m_2(\{y\})+a-m_2(E)$. Hence $\pi'=\pi-(a-m_2(E))\delta_x\otimes\delta_y$ is a non-negative measure and it is a coupling measure for $m_1+m_2(E)\delta_x$ and $m_2+m_1(E)\delta_y$. Since it is a restriction of $\pi$ and since optimality is inherited by restriction (see \cite[Theorem~4.6]{Villani2009}), it is an optimal coupling for its marginals. We deduce that
\begin{align*}
\cW_d(m_1+m_2(E)\delta_x,m_2+m_1(E)\delta_y)&=\int_{E\times E} d(u,v)\,\pi'(du,dv)\\
&=\int_{E\times E} d(u,v)\,\pi(du,dv) - (a-m_2(E))d(x,y)\\
&=\cW_d(m_1+a\delta_x,m_2+b\delta_y)-(a-m_2(E))d(x,y)).
\end{align*}
Subtracting $(m_1(E)+m_2(E))d(x,y)$ on both sides concludes the proof.
\end{proof}

\subsection{The particular case $N=1$}
\label{sec:N-equal-1}

In this section, we state our result in the simpler case $N=1$. The
following corollary is an immediate consequence of
Theorem~\ref{thm:main}.

\begin{cor}
\label{cor:only-1-part}
Let $L$ be the infinitesimal generator of a pure jump non-explosive
Markov process on $E$ defined, for any bounded measurable function
$f:E\rightarrow\R$, by
\begin{align*}
Lf(x)=\int_E (f(u)-f(x))\,q(x,du),\ \forall x\in E,
\end{align*}
where $(q(x,du))_{x\in E}$ is a jump kernel of finite non-negative
measures. Then the coarse Ricci curvature $\sigma$ of the Markov
process generated by $L$ satisfies
\begin{align*}
\sigma\geq -\sup_{x,y\in E} \frac{J_d^{x,y}\left(q(x,\cdot),q(y,\cdot)\right)}{d(x,y)}.
\end{align*}
\end{cor}

In Example~\ref{exa:tv-N-equal1}, we apply
Corollary~\ref{cor:only-1-part} to the case where $d$ is the
{\color{black} trivial distance} $d(x,y)=\11_{x\neq y}$ and
$q(x,\cdot)$ admits a density $\alpha(x,y)$ with respect to a common
non-negative measure $\zeta$ on $E$.  In Example~\ref{exa:bd-proc-1d},
we show that the lower bound obtained in
Corollary~\ref{cor:only-1-part} is in fact equal to the coarse Ricci
curvature in the case of birth and death processes. In a second
example, we compute a lower bound for a modified version of birth and
death processes, using a slight extension of a lemma by Vallender in
order to compute the Wasserstein distance between probability measures
on the real line.

{\color{black}
\begin{rem}
  For continuous time birth and death processes,
  Mielke~\cite{Mielke2013} recently computed a lower bound for an
  other notion of discrete Ricci curvature, related to the fact that
  the evolution of the law of a continuous time birth and death
  process can be described through a gradient flow system. To relate
  both definitions is still an open problem, but the lower bound
  obtained in Mielke's work has a similar expression (see Section~5 in
  \cite{Mielke2013} and Example~\ref{exa:bd-proc-1d} below) and may be
  a good starting point to compare both approaches. This example has
  also been considered by Fathi and Maas
  in~\cite[Theorem~4.1]{Fathi2016} in the setting of Entropic Ricci
  curvature.
\end{rem}}

\begin{exa}
\label{exa:tv-N-equal1} In this example, $d$ is the {\color{black} trivial} distance on $E$. Assume that there exist a non-negative measure $\zeta$ on $E$ and a measurable function $\alpha:E\times E\rightarrow\R_+$ such that 
\begin{align*}
q(x,dz)=\alpha(x,z)\,\zeta(dz),\ \forall x\in E.
\end{align*}
Without loss of generality, we assume that $\alpha(x,x)=0$ for all $x\in E$.
Then, using the fact that the Wasserstein distance is one half of the total variation distance, we obtain, for all $x\neq y\in E$,
\begin{align*}
J_d^{x,y}(q(x,\cdot),q(y,\cdot))&=\frac{1}{2} \int_{E\setminus\{x,y\}} |\alpha(x,z)-\alpha(y,z)|\,\zeta(dz)\\
&\quad+\frac{1}{2}|\alpha(x,x)\zeta(\{x\})+q(y,E)-\alpha(y,x)\zeta(\{x\})|\\
&\quad+\frac{1}{2}|\alpha(x,y)\zeta(\{y\})-\alpha(y,y)\zeta(\{y\})-q(x,E)| -q(x,E)-q(y,E)\\
&=\frac{1}{2}\int_{E\setminus\{x,y\}}|\alpha(x,z)-\alpha(y,z)|\,\zeta(dz)\\
&\quad-\frac{1}{2}\alpha(y,x)\zeta(\{x\})-\frac{1}{2}\alpha(x,y)\zeta(\{y\}) -\frac{1}{2}q(x,E)-\frac{1}{2}q(y,E),
\end{align*}
where we used the fact that $\alpha(x,x)\zeta(\{x\})+q(y,E)-\alpha(y,x)\zeta(\{x\})\geq 0$ and $\alpha(x,y)\zeta(\{y\})-\alpha(y,y)\zeta(\{y\})-q(x,E)\leq 0$. Rearranging the terms, we obtain
\begin{align*}
J_d^{x,y}(q(x,\cdot),q(y,\cdot))=-\int_{E} \alpha(x,z)\wedge\alpha(y,z)\,\zeta(dz)-\alpha(y,x)\zeta(\{x\})-\alpha(x,y)\zeta(\{y\}).
\end{align*}
In particular, the coarse Ricci curvature $\sigma$ of the process satisfies
\begin{align*}
\sigma\geq \inf_{x\neq y} \,\left[\int_{E} \alpha(x,z)\wedge\alpha(y,z)\,\zeta(dz)+\alpha(y,x)\zeta(\{x\})+\alpha(x,y)\zeta(\{y\})\right].
\end{align*}
\end{exa}

\begin{exa}
\label{exa:bd-proc-1d}
Consider the particular case where {\color{black}
  $E=\N^{0}:=\{0,1,2,\ldots\}$} and $L$ is the infinitesimal generator
of a birth and death process with birth rates
$(b_x)_{x\in\N^{\color{black}0}}$ and death rates $(d_x)_{x\in\N^{\color{black}0}}$, all
positive but $d_0=0$. In this case, for all $x,y\in\N^{\color{black}0}$,
\begin{align*}
q(x,{y})=\begin{cases}
b_x&\text{if }y=x+1\\
d_x&\text{if }x\geq 1\text{ and }y=x-1\\
0&\text{otherwise}
\end{cases}
\end{align*}
We also assume that the distance $d$ is given by
\begin{align*}
d(x,y)=\left|\sum_{k=0}^{x-1}u_k-\sum_{k=0}^{y-1}u_k\right|,
\end{align*}
where $(u_k)_{k\geq 0}$ is a sequence of positive numbers.
Using Proposition~\ref{prop:classical-coupling}, we obtain, for all $x\leq y-1$,
\begin{align*}
J_d^{x,y}(q(x,\cdot),q(y,\cdot))&\leq  d_x u_{x-1}-b_x u_x-d_y u_{y-1}+b_y u_y\\
&=\sum_{k=x}^{y-1} u_k\,\left(d_k \frac{u_{k-1}}{u_k}-b_k-d_{k+1}+b_{k+1}\frac{u_{k+1}}{u_k}\right)\\
&\leq -d(x,y)\,\inf_{x\in\N^{\color{black}0}} b_x+d_{x+1}-d_x \frac{u_{x-1}}{u_x}-b_{x+1}\frac{u_{x+1}}{u_x},
\end{align*}
with the convention $u_{-1}=0$.
Hence Corollary~\ref{cor:only-1-part} entails that the coarse Ricci curvature $\sigma$ of the process satisfies
\begin{align*}
\sigma\geq \inf_{x\in\N^{\color{black}0}} b_x+d_{x+1}-d_x \frac{u_{x-1}}{u_x}-b_{x+1}\frac{u_{x+1}}{u_x}.
\end{align*}
In \cite{Chen1994}, \cite{Joulin2009} and \cite{ChafaiJoulin2013}, it is shown that there is equality in the above equation. This implies that, at least in some cases, Corollary~\ref{cor:only-1-part} and hence Theorem~\ref{thm:main} are sharp. Note that, in this case, Proposition~\ref{prop:classical-coupling} provides an explicit expression for the quantity $J_d^{x,y}(q(x,\cdot),q(y,\cdot))$.
\end{exa}

\medskip
\begin{exa}
\label{exa:bd-proc-1d-modified}
The choice of the classical coupling (i.e. the use of Proposition~\ref{prop:classical-coupling}) in the previous example was judicious because the measures involved for a birth and death process are stochastically ordered : the jumps measures $q(x,\cdot)$ are such that $q(x,\cdot)+m_2(E)\delta_x$ is always dominated by $q(y,\cdot)+m_1(E)\delta_y$ for $x\leq y$, so that an optimal coupling between the measures involved is obtained by the classical coupling (this also explains why, in~\cite[Theorem~4.3]{Joulin2009} for instance, the  classical coupling is sufficient to recover the exact coarse Ricci curvature).   This is not the case in the present example.

We assume that
\begin{align*}
q(x,{y})=\begin{cases}
b_x&\text{if }y=x+2\\
d_x&\text{if }y=x-1,\\
0&\text{otherwise}.
\end{cases}
\end{align*}
 In this case, a similar computation as above shows that, for $x\leq y-2$,
\begin{align*}
J_d^{x,y}(q(x,\cdot),q(y,\cdot))\leq d_x u_{x-1}-b_x (u_x+u_{x+1})-d_y u_{y-1}+b_y(u_y+u_{y+1}).
\end{align*}
Using the same method (which relies on Proposition~\ref{prop:classical-coupling}) in the case $x=y-1$ would lead to the following bound
\begin{align*}
J_d^{x,y}(q(x,\cdot),q(y,\cdot))\leq d_x u_{x-1}-(b_x+d_{x+1})u_x+(b_x+b_{x+1})u_{x+1}+b_{x+1} u_{x+2}
\end{align*}
 Instead, we use Lemma~\ref{lem:vallender} below to obtain, when $x=y-1$,
\begin{align*}
\cW_d(q(x,\cdot)+q(y,E)&\delta_x,q(y,\cdot)+q(x,E)\delta_y)\\
&= d_x u_{x-1}+(d_x+b_{x+1})u_x+|b_{x+1}-b_x|u_{x+1}+b_{x+1} u_{x+2}
\end{align*}
and hence
\begin{align*}
J_d^{x,y}(q(x,\cdot),q(y,\cdot))=d_x u_{x-1}-(b_x+d_{x+1})u_x+|b_{x+1}-b_x|u_{x+1}+b_{x+1} u_{x+2}.
\end{align*}
Note that this quantity is always strictly smaller than the bound obtained using Proposition~\ref{prop:classical-coupling} (which corresponds to the classical coupling).
We deduce that the coarse Ricci curvature of the process satisfies
\begin{align*}
\sigma\geq \inf_{x\in\N^{\color{black}0}} b_x+d_{x+1}-d_x \frac{u_{x-1}}{u_x}-|b_{x+1}-b_x|\frac{u_{x+1}}{u_x}-b_{x+1}\frac{u_{x+2}}{u_x}.
\end{align*}
Lemma~\ref{lem:vallender}  allowed us to provide a computable bound for the coarse Ricci curvature. This method can be easily generalized to other jump measures on the real line and, although the coupling operator realizing this bound might be quite difficult to build explicitly, our result  shows that such a coupling operator indeed exists.
\end{exa}

\medskip

The following lemma, which is a slight extension of~\cite{Vallender1973}, can be useful to compute the Wasserstein distance between laws on the real line when the distance is similar to the one of the two previous examples. Note that in this statement, we define $\cW_d(m_1,m_2)$ as the infimum in~\eqref{eq:wass-def}, although $d$ might not be a distance in general. Of course, this result immediately extends to arbitrary non-negative measures $m_1$ and $m_2$ sharing the same mass.
\begin{lem}
\label{lem:vallender}
Let $\mu$ be a positive measure on $\R$ and consider the functional on $\R\times \R$ defined by $d(x,y)=\mu([\min(x,y),\max(x,y)))$ for all $x,y\in \R$. Then, for any probability measures $m_1$ and $m_2$ belonging to $\cM_d(\R)$, we have 
\begin{align*}
\cW_d(m_1,m_2)= \int_\R \left|F_1(t)-F_2(t)\right|\,\mu(dt),
\end{align*}
where $F_1$ and $F_2$ are the cumulative distribution functions of $m_1$ and $m_2$ respectively. Moreover, the infimum in the definition of $\cW_d(m_1,m_2)$ is attained.

\end{lem}

\begin{proof}[Proof of Lemma~\ref{lem:vallender}]
 Let $U$ be a random variable with uniform law on $(0,1)$ and define $X_1=F_1^{-1}(U)$ and $X_2=F_2^{-1}(U)$ where 
 \begin{align*}
 F_i^{-1}(p)=\min\{x\in\R,\ s.t.\ F_i(x)\geq p\},\ \forall p\in(0,1),\ i\in\{1,2\}.
 \end{align*}
 It is well known that the laws of $X_1$ and $X_2$ are $m_1$ and $m_2$
 respectively.  We consider the left-continuous non-decreasing
 function $f: x\mapsto \mu(-\infty,x)$ and define the random variables
 $Y_1=f(X_1)$ and $Y_2=f(X_2)$. We denote by $G_1$ and $G_2$ their
 respective cumulative distribution functions, and our first aim is to
 prove (in Step 1) that
\begin{align*}
G_i^{-1}(p)=f\circ F_i^{-1}(p)\ \forall p\in(0,1),\ i\in\{1,2\}.
\end{align*}
We conclude the proof of the lemma in Step~2, using a well known explicit expression for the Wasserstein distance between the laws of $Y_1$ and $Y_2$ when the underlying distance if the euclidean one.

\medskip\noindent
\textit{Step~1.} Fix $p\in(0,1)$ and $i\in\{1,2\}$ and let us prove that $G_i^{-1}(p)=f\circ F_i^{-1}(p)$.

We set
\begin{align*}
y_0:=G^{-1}_i(p)&=\min\left\{y\in\R\text{ such that }\P\left(f(F_i^{-1}(U))\leq y\right)\geq p\right\}.
\end{align*}
Since $p>0$, there exists $x\in\R$ such that $f(x)\leq y_0$, and, since $f$ is left continuous and non-decrea\-sing, we deduce that there exists  $x_0\in\R$ which is the largest number such that $f(x_0)\leq y_0 $.  Since, by definition of $x_0$, for all $x\in\R$, $f(x)\notin (f(x_0),y_0]$, we also observe that $$p\leq \P\left(f(F_i^{-1}(U))\leq y_0\right)=\P\left(f(F_i^{-1}(U))\leq  f(x_0)\right),$$ so that $y_0\leq f(x_0)$ by definition of $y_0$. Finally, we deduce that $y_0=f(x_0)$.

The definition of $x_0$ also entails that, for all $z\in \R$, $f(z)\leq y_0\Leftrightarrow z\leq x_0$. As a consequence,
\begin{align*}
p\leq \P\left(f(F_i^{-1}(U))\leq y_0\right)=\P(F_i^{-1}(U)\leq x_0)=F_i(x_0).
\end{align*}
We deduce that $x_0\geq F_i^{-1}(p)$ and hence that $G_i^{-1}(p)=f(x_0)\geq f\circ F_i^{-1}(p)$. 

Now, setting $x_0'= F_i^{-1}(p)$, we have $\P(F_i^{-1}(U)\leq x_0')=F_i(x_0')\geq p$ and hence $\P(f(F_i^{-1}(U))\leq f(x_0'))\geq p$ since $f$ is non-decreasing. This implies that $ f\circ F_i^{-1}(p)=f(x_0')\geq G_i^{-1}(p)$.

This concludes Step~1 of the proof.

\medskip\noindent\textit{Step 2.} Let us now conclude the proof of the lemma.

Denoting by $|\cdot|$ the euclidean distance on $\R$ and by $\cW_{|\cdot|}$ the corresponding Wasserstein distance, we obtain using~\cite{Vallender1973} (see also the Addendum by the same author in 1980 and references therein, see also~\cite{DallAglio1956} for an anterior look at the problem) that
\begin{align*}
\cW_{|\cdot|}\left(Law(Y_1),Law(Y_2)\right)
&=\E|Y_1-Y_2|=\int_\R \left|G_1(y)-G_2(y)\right|\,dy\\
&=\int_{\R_+} \left|G_1(y)-G_2(y)\right|\,dy,
\end{align*}
since $Y_1\geq 0$ and $Y_2\geq 0$ almost surely.
But, setting $g:y\mapsto \max\{x\in\R,\ f(x)\leq y\}$, we have, for all $y\in\R_+$ and $i\in\{1,2\}$,
\begin{align*}
G_i(y)=\P(f(X_i)\leq y)=\P(X_i\leq g(y))=F_i\circ g(y).
\end{align*}
Hence
\begin{align}
\label{eq:lemVallenderFin}
\E(d(X_1,X_2))=\E|Y_1-Y_2|=\int_{\R_+}  \left|F_1-F_2\right|\circ g(y)\,dy=\int_\R |F_1(t)-F_2(t)|\,\mu(dt).
\end{align}
In order to verify the last equality, one simply checks that, for all
$t\in\R$, the integral of the function $x\mapsto \11_{x < t}$ with
respect to the measure $A\mapsto \int_{\R_+} \11_A\circ g(y)\,dy$ is
\begin{align*}
\int_{\R_+} \11_{g(y)< t}\,dy=\int_{\R_+} \11_{y<f(t)}\,dy=f(t)=\int_\R \11_{x < t}\,\mu(dx).
\end{align*}
On the one hand, we deduce from~\eqref{eq:lemVallenderFin} that
\begin{align*}
\cW_d(m_1,m_2)\leq \int_\R |F_1(t)-F_2(t)|\,\mu(dt).
\end{align*}
On the other hand, for any coupling $(X'_1,X'_2)$ with marginal laws
$m_1$ and $m_2$, the coupling $(Y'_1,Y'_2)=(f(X'_1),f(X'_2))$ is a
coupling with the same marginal laws as $Y_1$ and $Y_2$. As a
consequence,
\begin{align*}
\E(d(X'_1,X'_2))=\E|Y'_1-Y'_2|\geq \cW_d(Law(Y_1),Law(Y_2))=\E|Y_1-Y_2|=\E(d(X_1,X_2)).
\end{align*}
This and Equation~\eqref{eq:lemVallenderFin} entail that 
\begin{align*}
\cW_d(m_1,m_2)= \int_\R |F_1(t)-F_2(t)|\,\mu(dt)
\end{align*}
and that the law of $(X_1,X_2)$ realizes the minimum in the definition of $\cW_d(m_1,m_2)$.
\end{proof}

\section{A model of interacting agents}
\label{sec:inter-agents}
In this section, the set $E$ is the complete graph of size $\# E\geq 3$ endowed with the distance $d(x,y)=\11_{x\neq y}$. In particular, the Wasserstein distance associated to $d$ equals half the total variation distance.

Fix $N\geq 2$ and consider the particle system described in the introduction, where each particle represents an agent's choice in the complete graph $E$. We recall that this model can be written in the settings of Theorem~\ref{thm:main}, by setting, for all $x,y\in E$,
\begin{align*}
F_i(x,\bar{x},\{y\})=\frac{T}{\# E}+f\left(\frac{\sum_{i=1}^N \11_{x_i=y}}{N}\right),\ \forall y\in E,
\end{align*}
where $f:[0,1]\rightarrow\R_+$ is a non-negative function and $T$ is a fixed constant (called the temperature of the system). Note that this process is exponentially ergodic, and that the marginal of its empirical stationary distribution is the uniform probability measure on $E$ (this is an immediate consequence of the symmetry of the state space and of the dynamic of the particles).

The following results are proved at the end of this section. In this first proposition, we assume that $f$ is Lipschitz and provide a  coarse Ricci curvature's lower bound that does not depend on $N$.
\begin{prop}
\label{prop:inter-agents-1}
Assume that $f$ is a Lipschitz function and define the Lipschitz constant of $f$ as $\|f\|_{Lip}=\sup_{u\neq v\in[0,1]} |f(u)-f(v)|/|u-v|$. Then the coarse Ricci curvature $\sigma$ of the particle system described above satisfies
\begin{align*}
\sigma\geq  T-{\color{black} 2\,}\|f\|_{Lip}+\inf_\mu \sum_{x\in E} f(\mu(x)),
\end{align*}
where the infimum is taken over the probability measures $\mu$ on $E$.  Moreover, if $f$ is monotone, then
\begin{align*}
\sigma\geq T-\|f\|_{Lip}+\inf_\mu \sum_{x\in E} f(\mu(x)).
\end{align*}
\end{prop}

In the next proposition, we assume that $f$ is a non-decreasing
strictly convex function and show that, for small values of $T$, the
process exhibits a meta-stable state, so that the agents have a herd
behavior for large values of $N$ : if all the agents start with the
same choice $x\in E$, then, during a time of order $\exp(cN)$, for
some constant $c>0$, $x$ is favored by the majority of the
agents. Note that this is true despite the fact that, during this very
same interval of time, the vast majority of the agents have changed
their choices at multiple times.

\begin{prop}
\label{prop:inter-agents-2}
Assume that $f$ is a strictly convex function such that $f(0)=0$, let $z_*\in(1/2,1)$ such that
\begin{align*}
z_*=\argmax_{z\in [1/2,1]} f(z)-z(f(z)+f(1-z))
\end{align*} 
and set
\begin{align*}
m_*=f(z_*)-z_*(f(z_*)+f(1-z_*))>0.
\end{align*}
If the temperature is sufficiently small, namely if
$$
0\leq T < \frac{m_*\# E}{z_*\# E-1},
$$
then there exists a positive constant $\delta>0$ such that, for all $x\in E$,
\begin{align*}
-\frac{1}{N}\log  \P\left(\exists s\in[0,t],\ \mu^N_s(x)\leq  z^*\right)  =_{N\rightarrow+\infty} \mathcal{O}\left(\min\left(\delta(\mu_0^N(x)-z_*)_+^2,\delta\bar{\varepsilon}^2-\frac{\log t}{N}\right)\right),
\end{align*}
uniformly in $t\geq 0$ and where $\mu^N_s=\frac{1}{N}\sum_{i=1}^N \delta_{X^i_s}$.
\end{prop}

In order to check that $m_*>0$ in the above result, one simply uses the fact that $f$ is strictly convex with $f(0)=0$, so that, for all $z\in(1/2,1)$, $f(1-z)/(1-z)<f(z)/z$.

\begin{exa}
  Assume that $f$ is an affine function : $f(x)=ax+b$ for some $a\in \R$ and $b\geq 0$ such that $a+b\geq 0$. Then $f$ is Lipschitz with $\|f\|_{Lip}=|a|$ and $\sum_{x\in E} f(\mu(x))=a+b\,\# E$ for any probability measure $\mu$ on $E$. Hence Proposition~\ref{prop:inter-agents-1} implies that the Wasserstein curvature of the process is bounded {\color{black} from} below by  $T+b\,\# E+a-|a|$. In particular, it is positive since
\begin{align*}
T+b\,\# E+a-|a|{\color{black} =\,}
\begin{cases}
T+b\# E>0&\text{if }a\geq 0,\\
T+b(\# E-2)+2(b+a)>0&\text{if }a<0,
\end{cases}
\end{align*}
 and hence the system of agents does not exhibit a herd behavior.
\end{exa}

\begin{exa}
Assume that $f(x)=x^2$. Then $\|f\|_{Lip}=2$ and 
\begin{align*}
\inf_{\mu} \sum_{x\in E}f(\mu(x))=\frac{1}{\# E},
\end{align*}
Moreover,
\begin{align*}
z_*=\argmax_{z\in[1/2,1]} z^2-z(z^2+(1-z)^2)=\frac{1}{2}+\frac{1}{\sqrt{12}}
\end{align*}
and
\begin{align*}
m_*=z_*^2-z_*(z_*^2+(1-z_*)^2)=\frac{1}{6\sqrt{3}}.
\end{align*}
Hence we deduce from Proposition~\ref{prop:inter-agents-1} and Proposition~\ref{prop:inter-agents-2} that
\begin{itemize}
\item if $T>{\color{black} 2\,}-1/\# E$, then the Wasserstein curvature
  of the particle system is positive (bounded {\color{black} from} below by
  $T-{\color{black} 2\,}+1/\# E$) and the system of agents does not
  exhibits a herd behavior;
\item if $0\leq T<\# E/((3+3\sqrt{3})\# E -6\sqrt{3})$, then the
  system of agents exhibits a herd behavior.
\end{itemize}
\end{exa}

\begin{proof}[Proof of Proposition~\ref{prop:inter-agents-1}]
Fix $i\in\{1,\ldots,N\}$ and $\bar{x},\bar{y}\in E^N$. We set $\mu_{\bar{x}}=\frac{1}{N}\sum_{j=1}^N\delta_{x_j}$ and  $\mu_{\bar{y}}=\frac{1}{N}\sum_{j=1}^N\delta_{y_j}$.  We assume, without loss of generality, that $F_i(x_i,\bar{x},E)\geq F_i(y_i,\bar{y},E)$. If $x_i\neq y_i$, one has
{\color{black}\begin{align*}
 \cW_d&\left(F_i(x_i,\bar{x},\cdot)+F_i(y_i,\bar{y},E)\delta_{x_i},F_i(y_i,\bar{y},\cdot)+F_i(x_i,\bar{x},E)\delta_{y_i}\right)\\
 &\quad= \frac{1}{2}\sum_{g\in E} \left|f\left(\mu_{\bar{x}}(g)\right)+\11_{g=x_i}\left(T+\sum_{h\in E} f\left(\mu_{\bar{y}}(h)\right)\right) -f\left(\mu_{\bar{y}}(g)\right)-\11_{g=y_i}\left(T+\sum_{h\in E} f\left(\mu_{\bar{x}}(h)\right)\right)\right|\\
&\quad\leq \frac{1}{2}\sum_{g\in E} \left|f\left(\mu_{\bar{x}}(g)\right)-f\left(\mu_{\bar{y}}(g)\right)\right|+\frac{1}{2}\sum_{g\in E} \left(f\left(\mu_{\bar{x}}(g)\right)-f\left(\mu_{\bar{y}}(g)\right)\right)+T+\sum_{g\in E} f\left(\mu_{\bar{y}}(g)\right).
\end{align*}}
If $x_i=y_i$, then 
\begin{align*}
 \cW_d&\left(F_i(x_i,\bar{x},\cdot)+F_i(y_i,\bar{y},E)\delta_{x_i},F_i(y_i,\bar{y},\cdot)+F_i(x_i,\bar{x},E)\delta_{y_i}\right)\\
 &\quad=  \frac{1}{2}\sum_{g\in E} \left|f\left(\mu_{\bar{x}}(g)\right)+\11_{g=x_i}\sum_{h\in E} f\left(\mu_{\bar{y}}(h)\right) -f\left(\mu_{\bar{y}}(g)\right)-\11_{g=y_i}\sum_{h\in E} f\left(\mu_{\bar{x}}(h)\right)\right|\\
&\quad\leq \frac{1}{2}\sum_{g\in E} \left|f\left(\mu_{\bar{x}}(g)\right)-f\left(\mu_{\bar{y}}(g)\right)\right|+\frac{1}{2}\sum_{g\in E} \left(f\left(\mu_{\bar{x}}(g)\right)-f\left(\mu_{\bar{y}}(g)\right)\right).
\end{align*}
In both expressions, we have
\begin{align}
\frac{1}{2}\sum_{g\in E} &\left|f\left(\mu_{\bar{x}}(g)\right)-f\left(\mu_{\bar{y}}(g)\right)\right|+\frac{1}{2}\sum_{g\in E} \left(f\left(\mu_{\bar{x}}(g)\right)-f\left(\mu_{\bar{y}}(g)\right)\right)\nonumber\\
&\quad\quad\quad\quad\quad\quad=\sum_{g\in E} \left(f\left(\mu_{\bar{x}}(g)\right)-f\left(\mu_{\bar{y}}(g)\right)\right)_+\nonumber\\
&\quad\quad\quad\quad\quad\quad\leq \|f\|_{Lip}\sum_{g\in E} \left|\mu_{\bar{x}}(g)-\mu_{\bar{y}}(g)\right| \leq {\color{black} 2\,}\|f\|_{Lip}d(\bar{x},\bar{y}) 
\label{eq:2-f-lip}
\end{align}
and hence, since $d(x_i,y_i)=1$ in the case $x_i\neq y_i$ and $d(x_i,y_i)=0$ in the case $x_i=y_i$, we deduce that
  \begin{align*}
    J_d^{x_i,y_i}(F_i(x_i,\bar{x},\cdot),F_i(y_i,\bar{y},\cdot))
    &\leq -F_i(x_i,\bar{x},E)d(x_i,y_i)+{\color{black} 2\,}\|f\|_{Lip} d(\bar{x},\bar{y}).
\end{align*}
We deduce that
\begin{align*}
\frac{\frac{1}{N}\sum_{i=1}^N J_d^{x_i,y_i}(F_i(x_i,\bar{x},\cdot),F_i(y_i,\bar{y},\cdot))}{d(\bar{x},\bar{y})}\leq -\inf_{i,\bar{x},\bar{y}} F_i(x_i,\bar{x},E)+{\color{black} 2\,}\|f\|_{Lip}.
\end{align*}
Since $\inf_{i,\bar{x},\bar{y}} F_i(x_i,\bar{x},E)=T+\inf_{\mu}\sum_{x\in E} f(\mu(x))$, this concludes the first part of the proof of Proposition~\ref{prop:inter-agents-1}.

If $f$ is non-decreasing (and similarly if $f$ is decreasing), then one can replace the inequality~\eqref{eq:2-f-lip} by (we use the fact that $\sum_{g\in E} \mu_{\bar{x}}(g)-\mu_{\bar{y}}(g)=0$)
\begin{align*}
\sum_{g\in E} \left(f\left(\mu_{\bar{x}}(g)\right)-f\left(\mu_{\bar{y}}(g)\right)\right)_+&\leq \|f\|_{Lip}\sum_{g\in E} \left(\mu_{\bar{x}}(g)-\mu_{\bar{y}}(g)\right)_+\\
&=\frac{1}{2}\|f\|_{Lip}\sum_{g\in E} \left|\mu_{\bar{x}}(g)-\mu_{\bar{y}}(g)\right|\\
&\leq \|f\|_{Lip}d(\bar{x},\bar{y}),
\end{align*}
which, as above, allows to conclude the proof of the second part of Proposition~\ref{prop:inter-agents-1}.
\end{proof}

\begin{proof}[Proof of Proposition~\ref{prop:inter-agents-2}]
 The particle system is a mean-field particle system and hence his empirical measure process, defined as
\begin{align*}
\mu^N_t=\frac{1}{N}\sum_{i=1}^N\delta_{X^i_t}\quad\forall t\geq 0,
\end{align*}
is a Markov process evolving in the simplex of $\R^{\# E}$.

Denote by $(\mu_t)_{t\geq 0}$ the solution to the ODE
\begin{align*}
\frac{d\mu_t}{dt}=\mu_t\left(L_{\mu_t}\cdot\right),\ \mu_0=\mu^N_0(x),
\end{align*}
with $L_{\mu}$ defined as the following operator acting on functions $h:E\rightarrow\R$ 
\begin{align*}
L_{\mu}h(k)=\sum_{l\in E} (h(l)-h(k))\left(\frac{T}{\#E}+f(\mu(l))\right).
\end{align*}
Then $(\mu^N_t)_{t\geq 0}$ satisfies the following upper bound large deviation principle proved in
\cite{DupuisEllisEtAl1991}, where we use the fact that the state space is compact (we also refer the reader to the more recent
\cite{DupuisRamananEtAl2016} for more general mean-field interactions with multiple particles jumps) : for any closed
set $F\in D([0,t],\R^{\# E})$, and all $\eta>0$, there exists $N_\eta\geq 2$ such that, for all $n\geq N_\eta$,
\begin{align}
\label{eq:upper-bound}
 \frac{1}{N}\log\P_{\mu_0^N}\left(\left(\mu^N_s\right)_{s\in [0,t]}\in F\right)\leq -\inf_{\varphi\in F} I_{\mu_0^N}(\varphi)+\eta,
\end{align}
where $I_{\mu_0^N}(\varphi)=+\infty$ if $\varphi(0)\neq \mu^N_0$ or if $\varphi$ is not absolutely continuous, and, otherwise,
\begin{align*}
I_{\mu_0^N}(\varphi)=\int_0^t \ell\left(\varphi_s,\frac{\d\varphi_s}{\d s}\right)\,ds,
\end{align*}
with $\ell$ defined as
\begin{align*}
\ell(\mu,\nu)=\sup_{\alpha\in \R^{\# E}} \alpha\cdot\nu-\sum_{k\in E} \mu(k) \sum_{l\in E} \left(\frac{T}{\# E}+f(\mu(l))\right)\left(\exp\left(\alpha_l-\alpha_k\right)-1\right).
\end{align*}
In the following, we choose $t=1$ in the definition of $I_{\mu_0^N}$.

\textit{Step 1:} Our first aim is to prove that there exists a constant $C\geq 1$ such that
\begin{align}
\label{eq:step1}
\ell(\mu,\nu)&\geq \begin{cases}
\frac{\|\nu-\mu(L_\mu\cdot)\|^2}{4C}&\text{if }\|\nu-\mu(L_\mu\cdot)\|\leq 2C,\\
\|\nu-\mu(L_\mu\cdot)\|-C&\text{if }\|\nu-\mu(L_\mu\cdot)\|\geq 2C
\end{cases}\\
&\geq \frac{\|\nu-\mu(L_\mu\cdot)\|^2}{4C}\wedge \frac{\|\nu-\mu(L_\mu\cdot)\|}{2},
\end{align}
where $\|\cdot\|$ denotes the Euclidean norm.
The main difficulty is that we require $C$ to be independent of $\mu$. Otherwise, the property would be directly obtained from the fact that $l\left(\mu,\nu\right)$ is strictly convex in its second variable, which is a consequence of \cite[Theorem 12.2]{Rockafellar1970}, as stressed out by \cite[Lemma~7.2]{DupuisRamananEtAl2016}. 

In the case where $\nu=\mu(L_\mu\cdot)$, we have
\begin{align*}
\ell\left(\mu,\mu(L_\mu\cdot)\right)=\sup_{\alpha\in \R^{\# E}} \sum_{k\in E} \mu(k) \sum_{l\in E} \left(\frac{T}{\# E}+f(\mu(l))\right)\left(-\exp\left(\alpha_l-\alpha_k\right)+1+(\alpha_l-\alpha_k)\right)=0.
\end{align*}
Now, if $\nu=\mu(L_\mu\cdot)+\zeta$, then, using the fact that $(\exp(z)-1-z)/z^2$ is uniformly bounded over $z\in[-1,1]$,
\begin{align*}
\ell\left(\mu,\nu\right)&\geq  \sup_{\alpha\in\R^{\#E},\|\alpha\|\leq 1}\alpha\cdot\zeta-C\|\alpha\|^2
\end{align*}
where $C\geq 1$ is a constant that does not depend on $\mu,\nu$ nor $\alpha$. If $\|\zeta\|\leq 2C$, then one can choose $\alpha=\zeta/2C$ in order to obtain $\ell\left(\mu,\nu\right)\geq \frac{\|\zeta\|^2}{4C}$.
If $\|\zeta\|\geq 2C$, then one can choose $\alpha=\zeta/\|\zeta\|$ and obtain
$
\ell\left(\mu,\nu\right)\geq -C +\|\zeta\|.
$
Finally, we deduce that~\eqref{eq:step1} holds true.

\textit{Step 2:} Our aim is to prove that there exists $\varepsilon^*>0$ and $\bar{\varepsilon}\in(0,\varepsilon^*/2]$
such that $\mu(L_\mu\11_x)\geq \bar{\varepsilon}$ for all probability measure $\mu$ on $E$ such that
$\mu(x)\in[z^*,z^*+\varepsilon^*]$.

Let $\mu$ be a probability measure on $E$ such that $\mu(x)= z_*+\varepsilon$ for some $\varepsilon>0$. We have
\begin{align*}
\mu\left(L_\mu\11_{x}\right)&=\sum_{y\in E} \mu(y)\sum_{k\in E} \left(\11_{x}(k)-\11_{x}(y)\right)\left(\frac{T}{\#E}+f(\mu(k))\right)\\
&=\frac{T}{\# E}+f(\mu(x))-\mu(x)\left(T+\sum_{k\in g} f(\mu(k))\right)
\end{align*}
Since $f$ is convex with $f(0)=0$, we have
\begin{align*}
\sum_{k\in E} f(\mu(k))\leq f(\mu(x))+f(1-\mu(x)).
\end{align*}
Hence
\begin{align*}
\mu\left(L_\mu\11_{x}\right)&\geq \frac{T}{\# E}+f(\mu(x))-\mu(x)\left(T+f(\mu(x))+f(1-\mu(x))\right)\\
&=  \frac{T}{\# E}-(z_*+\varepsilon) T+f(z_*+\varepsilon)-(z_*+\varepsilon)(f(z_*+\varepsilon)+f(1-z_*-\varepsilon)).
\end{align*}
Now, since the right hand side of the above term is continuous in $\varepsilon$ and strictly positive when
$\varepsilon=0$ (by assumption on $T$), we conclude that there exists two positive constants $\varepsilon^*$ and
$\bar{\varepsilon}$ such that the above term is larger than $\bar{\varepsilon}$ for all
$\varepsilon\in[0,\varepsilon^*]$. Since one can assume without loss of generality that
$\bar{\varepsilon} \leq \varepsilon^*/2$, this concludes Step~2.

\textit{Step 3:} Let $\mu_0^N$ be such that $\mu_0^N(x)>z_*$. Using the results of the previous steps, we show that any function
$\varphi_\cdot\in D([0,1],\R^{\# E})$ with values in the simplex, such that $\varphi_0(x)>z^*$ and such that $\varphi_t(x)<z^*+t\bar{\varepsilon}$ for some $t\in[0,1]$,
satisfies
\begin{align*}
I_{\mu_0^N}(\varphi)\geq \delta(\mu_0^N(x)-z_*)^2
\end{align*}
for some constant $\delta>0$.
 
Consider $\varphi$ satisfying the above property. If $\varphi(x)\neq \mu_0^N(x)$ or if $\varphi$ is not absolutely continuous,
then $I_{\mu_0^N}(\varphi)=+\infty$ and the property is immediate. Otherwise, there exist two times $t_1<t_2\in[0,1]$,
such that $\varphi_{t_1}(x)=\varphi_0(x)\wedge (z^*+\bar{\varepsilon})+t_1\bar{\varepsilon}$, $\varphi_{t_2}(x)=z_*+t_2\bar{\varepsilon}$ and
$\varphi_s(x)\in[z^*,z^*+2\bar{\varepsilon}]$ for all $s\in[t_1,t_2]$. In particular, Step~1 entails
\begin{align*}
I(\varphi)&\geq \int_{t_1}^{t_2} \frac{\left\|\varphi_s(L_{\varphi_s}\cdot)-\frac{\d\varphi_s}{\d s}\right\|^2}{4C}\wedge \frac{\left\|\varphi_s(L_{\varphi_s}\cdot)-\frac{\d\varphi_s}{\d s}\right\|}{2}\,ds\\
&\geq \frac{1}{4C} \int_{t_1}^{t_2} \left|\varphi_s(L_{\varphi_s}\11_x)-\frac{\d\varphi_s}{\d s}(x)\right|^2\wedge \left|\varphi_s(L_{\varphi_s}\11_x)-\frac{\d\varphi_s}{\d s}(x)\right|\,\11_{\frac{\d\varphi_s}{\d s}(x)\leq \bar{\varepsilon}}ds\\
&\geq \frac{1}{4C}\int_{t_1}^{t_2} \left|\bar{\varepsilon}-\frac{\d\varphi_s}{\d s}(x)\right|^2\wedge \left|\bar{\varepsilon}-\frac{\d\varphi_s}{\d s}(x)\right|\,\11_{\frac{\d\varphi_s}{\d s}(x)\leq \bar{\varepsilon}}ds,
\end{align*}
since Step~2 implies that $\varphi_s(L_{\varphi_s}\11_x)\geq \bar{\varepsilon}$ for all $s\in[t_1,t_2]$. Setting 
\[
A=\left\{s\in[t_1,t_2],\ \text{s.t. }\left|\bar{\varepsilon}-\frac{\d\varphi_s}{\d s}(x)\right|^2\leq \left|\bar{\varepsilon}-\frac{\d\varphi_s}{\d s}(x)\right|\right\},
\]
and using Cauchy-Schwarz inequality, we obtain
\begin{align*}
  \int_{t_1}^{t_2} \left|\bar{\varepsilon}-\frac{\d\varphi_s}{\d s}(x)\right|^2\11_{A}(s)\,\11_{\frac{\d\varphi_s}{\d s}(x)\leq \bar{\varepsilon}}ds\geq \left(\int_{t_1}^{t_2} \left|\bar{\varepsilon}-\frac{\d\varphi_s}{\d s}(x)\right|\11_{A}(s)\,\11_{\frac{\d\varphi_s}{\d s}(x)\leq \bar{\varepsilon}}ds\right)^2/(t_2-t_1).
\end{align*}
But
\begin{multline*}
  \int_{t_1}^{t_2} \left|\bar{\varepsilon}-\frac{\d\varphi_s}{\d s}(x)\right|\11_{A}(s)\,\11_{\frac{\d\varphi_s}{\d s}(x)\leq \bar{\varepsilon}}ds+\int_{t_1}^{t_2} \left|\bar{\varepsilon}-\frac{\d\varphi_s}{\d s}(x)\right|\11_{A^c}(s)\,\11_{\frac{\d\varphi_s}{\d s}(x)\leq \bar{\varepsilon}}ds\\
  \geq \bar{\varepsilon}t_2-\varphi_{t_2}(x)-\bar{\varepsilon}t_1+\varphi_{t_1}(x)=\varphi_0(x)\wedge(z_*+\bar{\varepsilon})-z_*,
\end{multline*}
hence one of the two terms in the left hand side is larger than $\frac{(\varphi_0(x)-z_*)\wedge \bar{\varepsilon}}{2}$, so that
\begin{align*}
  \int_{t_1}^{t_2} \left|\bar{\varepsilon}-\frac{\d\varphi_s}{\d s}(x)\right|^2\wedge \left|\bar{\varepsilon}-\frac{\d\varphi_s}{\d s}(x)\right|\,\11_{\frac{\d\varphi_s}{\d s}(x)\leq \bar{\varepsilon}}ds
  &\geq \left(\frac{(\varphi_0(x)-z_*)\wedge \bar{\varepsilon}}{2\sqrt{t_2-t_1}}\right)^2\wedge\frac{(\varphi_0(x)-z_*)\wedge \bar{\varepsilon}}{2}\\
  &\geq \delta (\varphi_0(x)-z^*)^2
\end{align*}
for some constant $\delta>0$. This concludes Step~3.

\textit{Step~4.} We conclude the proof by a classical renewal argument. Using
Step~3, the deviation principle~\eqref{eq:upper-bound} and using the Markov property, one obtains
\begin{align*}
  \P\left(\exists s\in[0,1],\ \mu^N_s(x)\leq  z^*+s\bar{\varepsilon}\right)&\leq \exp\left(-N(\delta(\mu_0^N(x)-z_*)^2-\eta)\right)
\end{align*}
and, for all integer $t\geq 1$,
\begin{multline*}
  \P\left(\exists s\in[t,t+1],\ \mu^N_s(x)\leq
    z^*+(s-t)\bar{\varepsilon}\mid\, \forall \mu^N_t(x)\geq
    z^*+\bar{\varepsilon}\right) \leq
  \exp\left(-N(\delta\bar{\varepsilon}^2-\eta)\right).
\end{multline*}
Hence, for all $t> 0$,
\begin{multline*}
  \P\left(\exists s\in[0,t],\ \mu^N_s(x)\leq  z^*\right)\\
  \leq \exp\left(-N(\delta(\mu_0^N(x)-z_*)^2-\eta)\right)+(\lceil t\rceil -1)\exp\left(-N(\delta\bar{\varepsilon}^2-\eta)\right).
\end{multline*}
Since $\eta$ can be chosen arbitrarily small uniformly in $t$, taking the logarithm allow
us to conclude the proof of Proposition~\ref{prop:inter-agents-2}.

\end{proof}

\section{Application to other models}
\label{sec:other-models}

In this Section, we compute a lower bound for the coarse Ricci curvature of different interacting particle systems. In Subsection~\ref{sec:zero-range-dyn}, we consider zero range dynamics.  In Subsection~\ref{sec:Fleming-Viot}, we study the case of Fleming-Viot type systems and some natural extensions. In Subsection~\ref{sec:BD-mean-field}, we consider birth and death processes in mean-field type interaction. Finally, we conclude in Subsection~\ref{sec:with-density} with systems of particles whose jump measures admit a density with respect to the Lebesgue measure or the counting measure (we consider exponential laws on $\R$, Gaussian measures on $\R^d$ or finitely supported discrete measures on {\color{black} $\Z$}). For the sake of clarity, we chose $F=F_i$ independent of $i$, but the approach and most computations remain unchanged in the dependent case.

\subsection{Zero range dynamics}
\label{sec:zero-range-dyn}

Let $E$ be a finite {\color{black} or countable space equipped with the trivial distance (defined by
$d(x,y)=\11_{x\neq y}$ for all $x,y\in E$)} and let $(P_{xy})_{x,y\in E}$ be a stochastic
matrix and consider a particle system whose infinitesimal generator is
given by
\begin{align}
\label{eq:infinit-gen-zero-rng}
\color{black} \cL f(\bar{x})=\sum_{i=1}^N c_{x_i}(\bar{x})\sum_{y\in E} P_{xy}(f(\bar{x}+e_y-e_x)-f(x)),
\end{align}
where, for all $x\in E$, $c_x:E^N\rightarrow\R_+$ is a non-negative
function. The particles of this system jump with respect to the
transition probability $P_x$ at a rate $c_{x_i}(\bar{x})$.

In the following corollary, for all $x,y\in E$, $\theta_P(x,y)$ is the
coarse Ricci curvature (in discrete time) of the transition
probability matrix $P$ along $(xy)$, in the sense
of~\cite[Definition~3]{Ollivier2009}:
\begin{align}
\label{eq:CRC-discrete-time}
\theta_P(x,y)=1-\frac{\cW_d(P_x,P_y)}{d(x,y)}.
\end{align}
{\color{black} In our case, $d$ is the trivial distance,} $\theta_P(x,y)$ is thus
equal to $1-\frac{1}{2}\|P_x-P_y\|_{TV}$. We refer the reader
to~\cite{Ollivier2009,Ollivier2010}, where the author provides general
properties and explicit bounds of $\theta_P$ for several choices of
$P$. {\color{black}In the following result, we abusively write
  $x\in\bar x$ if $x$ is a coordinate of $\bar x$.}

\begin{cor}
\label{cor:zero-range}
The coarse Ricci curvature $\sigma$ of the particle system with infinitesimal generator $\cL$ given by~\eqref{eq:infinit-gen-zero-rng} satisfies
\begin{align*}
\sigma\geq \inf_{\color{black}\bar{x}\neq \bar{y}\in E^N,\ x\in\bar x,\,y\in\bar y} \theta_P(x,y) c_x(\bar{x})\wedge c_y(\bar{y})-\sup_{x\in E} (1-P_{xx})\|c_x\|_{Lip}.
\end{align*}
where $\|c_x\|_{Lip}:=\sup_{\color{black}\bar{x}\neq \bar{y}\in E^N\text{ s.t. }x\in \bar x\text{ and }x\in\bar y} \frac{|c_x(\bar{x})-c_x(\bar{y})|}{d(\bar{x},\bar{y})}$.
\end{cor}

{\color{black} One says that the infinitesimal
  generator~\eqref{eq:infinit-gen-zero-rng} defines a zero range
  dynamic, if the jump rate and distribution of one particle does not
  depend on the position of the particles located at other sites. This
  means that, for each $x\in E$, there exists a non-negative function
  $c^{\rm zr}_x:\N\rightarrow\R_+$, where $\N:=\{1,2,\ldots\}$, such that the infinitesimal
  generator of the particle system on $E^N$ is given by
  \begin{align}
     \label{eq:infinit-gen-zero-rng-true}
    \cL^{\rm zr} f(\bar{x})=\sum_{i=1}^N c^{\rm zr}_{x_i}(\eta^{\bar x}_{x_i})\frac{1}{\# E}\sum_{y\in E} P_{{x_i}y}(f(\bar{x}+e_y-e_{x_i})-f(x_i)),
  \end{align}
  where $\eta^{\bar x}_x\stackrel{def}{=}\sum_{i=1}^N \11_{x_i=x}$
  denotes the number of components of $\bar{x}$ equal to $x$.  In this
  situation and when $d$ is the trivial distance on $E$, the bound
  obtained in the above corollary can be refined, as stated in the
  following result.

  \begin{cor}
    \label{cor:zero-range-true}
    The coarse Ricci curvature $\sigma^{\rm zr}$ of the zero range
    particle system with infinitesimal generator $\cL^{\rm zr}$ given
    by~\eqref{eq:infinit-gen-zero-rng-true} satisfies
    \begin{align*}
      \sigma^{\rm zr}\geq \inf_{\substack{x\in E, y\in E\\ n,m\in \N}} \theta_P(x,y) c^{\rm zr}_x(n)\wedge c^{\rm zr}_y(m)-2\sup_{\substack{x\in E\\n,m\in\N}} (1-P_{xx})\frac{n}{m}|c^{\rm zr}_x(n)-c^{\rm zr}_x(n+m)|.
    \end{align*}
  \end{cor}

  An interesting feature of a zero range dynamic is that the empirical
  measure of the process is a measure valued Markov process, whose
  coarse Ricci curvature is bounded {\color{black} from} below by the
  coarse Ricci curvature of the dynamic of the full particle system in
  $E^N$.  In~\cite[Section~4]{CaputoEtAl2009}, the authors study the
  mixing properties of the empirical measure dynamic, with the
  assumption that $E$ is finite, that $P_x$ is the uniform measure on
  $E$ for all $x\in E$ and that there exist $0\leq\delta\leq c$ such
  that, for all $x\in E$,
  \[
    c\leq (n+1)c^{\rm zr}_x(n+1)-n c^{\rm zr}_x(n) \leq c+\delta,\quad \forall n\geq 1.
  \]
  (beware that the rate denoted by $c_x(n)$ in~\cite{CaputoEtAl2009}
  is the jump rate for $n$ particles, and hence it corresponds to
  $n\,c^{\rm zr}_x(n)$ in our settings). Under this set of
  assumptions, they obtain a \textit{modified logarithmic Sobolev
    inequality} with rate $c-\delta$, which provides a lower bound for
  the rate of exponential convergence to equilibrium of the process,
  in the relative entropy sense. In~\cite{Boudou2006}, the authors prove that
  the spectral gap for the empirical measure process is lower bounded
  by $c>0$ under weaker assumptions (namely with $\delta=+\infty$).

  Under this particular set of assumptions and considering $d$ equal
  to the trivial distance on $E$, one has $\theta_P(x,y)=1$,
  $P_{xx}=1/\# E$, $c_x^{\rm zr}\geq c$ and
  $\frac{n}{m}|c^{\rm zr}_x(n)-c^{\rm zr}_x(n+m)|\leq \delta$ for all
  $x\in E, y\in E, n,m\in \N$. Hence
  Corollary~\ref{cor:zero-range-true} implies that $c-2\delta$ is a
  lower bound for the coarse Ricci curvature of the particle
  system. 

  As expected, we obtain a weaker lower bound for the rate of
  convergence to the equilibrium of this zero range dynamic than
  in~\cite{CaputoEtAl2009}, since we consider the dynamic of the full
  particle system instead of its empirical measure. However, it is
  interesting to note that both bounds share a similar structure. Note
  also that, contrarily to~\cite{CaputoEtAl2009}, we do not require
  the functions $n\mapsto n\,c_x(n)$ to be non-decreasing and hence
  provide a new result for the convergence rate to equilibrium of such
  zero-range dynamics (both for the full particle system and for the
  empirical measure). Of course, one expects that the actual rate of
  convergence to equilibrium for the empirical measure is higher than
  the one we found.

  Finally, one interesting aspect of our result on $\cW_d$ convergence
  is that we do not require that the process is reversible, allowing
  various choices of $(P_x)_{x\in E}$.


\begin{rem}
  Entropy Ricci curvature of zero range dynamic models have also been
  studied by Fathi and Maas in~\cite[Section~4.2]{Fathi2016}. Under
  the assumptions of~\cite[Section~4]{CaputoEtAl2009} and that
  $\delta\in[0,2c]$, the authors prove that the Entropy Ricci
  curvature is lower bounded by $c/2-5\delta/4$ and hence that this
  system satisfies a gradient flow structure with positive curvature
  when $c>5\delta/2$.
\end{rem}}

\begin{proof}[\color{black}Proof of Corollary~\ref{cor:zero-range}]
With the notation of Theorem~\ref{thm:main}, the jump measures of this interacting particle system are
\begin{align*}
F(x,\bar{x},\cdot)=c_{x}(\bar{x})P_x,\ \forall x\in E,\ \bar{x}\in E^N.
\end{align*}
Using properties~\eqref{eq:J-prop1} and~\eqref{eq:J-prop2}, we obtain  for all $x, y\in E$ and $\bar{x},\bar{y}\in E^N$ such that $c_x(\bar{x})\geq c_y(\bar{y})$
\begin{align*}
J^{x,y}_d(F(x,\bar{x},\cdot),F(y,\bar{y},\cdot))&\leq 
c_y(\bar{y})\,J^{x,y}_d\left(P_x,P_y\right)+(c_x(\bar{x})-c_y(\bar{y}))\,J^{x,y}_d\left(P_x,0\right)\\
&\leq c_{y}(\bar{y})\left[\cW_d\left(P_x,P_y\right)-d(x,y)\right]+(c_x(\bar{x})-c_y(\bar{y}))\left[\cW_d(P_x,\delta_y)-d(x,y)\right]\\
&\leq -c_{y}(\bar{y})\theta_P(x,y)+ (c_x(\bar{x})-c_y(\bar{y}))(1-P_{xx})\11_{x=y}.
\end{align*}
As a consequence, for all $\bar{x},\bar{y}\in E^N$,
\begin{align}
\sum_{i=1}^N J^{x_i,y_i}_d(F(x_i,\bar{x},\cdot),F(y_i,\bar{y},\cdot))
&\leq - \sum_{i=1}^N \11_{x\neq y}\theta_P(x_i,y_i)c_{x_i}(\bar{x})\wedge c_{y_i}(\bar{y})\nonumber\\
&\phantom{\leq - \sum_{i=1}^N }+ \sum_{i=1}^N \11_{x_i=y_i} (1-P_{x_ix_i}) \left|c_{x_i}(\bar{x})-c_{y_i}(\bar{y})\right|\label{eq:useful11}\\
&\leq -N\,d(\bar{x},\bar{y})\,\inf_{\color{black}\bar{u}\neq \bar{v}\in E^N,\,u\in\bar{u},v\in\bar{v}} \theta_P(u,v)c_{u}(\bar{u})\wedge c_{v}(\bar{v})\nonumber\\
&\phantom{\leq (1-}+N\,d(\bar{x},\bar{y}) \sup_{\color{black}\substack{u\in E,\bar{u}\neq \bar{v}\in E^N\\u\in\bar{u}\text{ and }u\in\bar{v}}} (1-P_{uu})\frac{|c_u(\bar{u})-c_u(\bar{v})|}{d(\bar{u},\bar{v})}.\nonumber
\end{align}
This and Theorem~\ref{thm:main} allow us to conclude the proof.
\end{proof}

{\color{black}
\begin{proof}[Proof of Corollary~\ref{cor:zero-range-true}]
  The same calculations as above up to~\eqref{eq:useful11} lead to
  \begin{align*}
    \sum_{i=1}^N J^{x_i,y_i}_d(F(x_i,\bar{x},\cdot),F(y_i,\bar{y},\cdot))
    &\leq -N\,d(\bar{x},\bar{y})\,\inf_{\color{black}\bar{u}\neq \bar{v}\in E^N,\,u\in\bar{u},v\in\bar{v}} \theta_P(u,v)c_{u}(\bar{u})\wedge c_{v}(\bar{v})\\
    &\phantom{\leq - \sum_{i=1}^N \11_{x\neq y}}+ \sum_{x\in E}\sum_{i=1}^N \11_{x_i=y_i=x} (1-P_{xx}) \left|c^{\rm zr}_{x}(\eta^{\bar x}_x)-c^{\rm zr}_{x}(\eta^{\bar{y}}_x)\right|.
  \end{align*}
  But $\sum_{i=1}^N \11_{x_i=y_i=x}\leq \eta^{\bar x}_x\wedge\eta^{\bar y}_x$, so that
  \begin{multline*}
    \sum_{x\in E}\sum_{i=1}^N \11_{x_i=y_i=x} (1-P_{xx}) \left|c^{\rm zr}_{x}(\eta^{\bar x}_x)-c^{\rm zr}_{x}(\eta^{\bar{y}}_x)\right|
    \leq \sum_{x\in E}\eta^{\bar x}_x\wedge\eta^{\bar y}_x (1-P_{xx}) \left|c^{\rm zr}_{x}(\eta^{\bar x}_x)-c^{\rm zr}_{x}(\eta^{\bar{y}}_x)\right|\\
    \leq \sum_{x\in E}|\eta^{\bar x}_x-\eta^{\bar y}_x| \,\sup_{\substack{u\in E\\n,m\in\N}} (1-P_{uu})\frac{n}{m}|c^{\rm zr}_u(n)-c^{\rm zr}_u(n+m)|.
  \end{multline*}
  Observing that
  $\sum_{x\in E}|\eta^{\bar x}_x-\eta^{\bar y}_x|\leq 2\,N\,d(\bar x,\bar
  y)$, one can use Theorem~\ref{thm:main} to conclude the proof.
\end{proof}}

\subsection{Some simple variants of Fleming-Viot type systems}
\label{sec:Fleming-Viot}

Assume that the distance $d$ is bounded by $d_\infty$ over $E\times E$. We consider the situation where there exist a measurable function $\beta:E\rightarrow\R_+$ and a Markovian kernel $(P_x)_{x\in E}$  such that
\begin{align}
\label{eq:F-for-ZR-FV}
F(x,\bar{x},dz)=q(x,dz)+\frac{\beta(x)}{N}\sum_{i=1}^N P_{x_i}(dz)=q(x,dz)+\beta(x) \mu_{\bar{x}} P,
\end{align}
where $\mu_{\bar{x}}=\frac{1}{N}\sum_{i=1}^N \delta_{x_i}$ is the empirical distribution of $\bar{x}$. In opposition  to the zero range dynamics of the previous subsection, the jump rate of a particle only depends on its position and its jump measure depend on the whole position of the system.

Note that, in the case where $P_x=\delta_x$ for all $x\in E$, we recover the Fleming-Viot type system introduced in~\cite{BurdzyHolystEtAl1996,DelMoral2013} and whose coarse Ricci curvature with respect to the {\color{black} trivial} distance $d(x,y)=\11_{x\neq y}$ has been studied in~\cite{CloezThai2016} (see also \cite{AsselahFerrariEtAl2011,
DelMoral2013,GrigorescuKang2004,
BurdzyHolystEtAl2000,Villemonais2014} for general properties).

Setting $\theta_*=\inf_{x\neq y\in E} \theta_P(x,y)\in[0,1]$, where $\theta_P(x,y)$ is the discrete time coarse Ricci curvature of $P$ (see Subsection~\ref{sec:zero-range-dyn}), we have, for all $x,y\in E$ such that $\beta(x)\geq \beta(y)$,
\begin{align*}
J^{x,y}_d(F(x,\bar{x},\cdot),F(y,\bar{y},\cdot))&\leq J_d^{x,y}(q(x,\cdot),q(y,\cdot))+\beta(y)\cW_d(\mu_{\bar{x}}P,\mu_{\bar{y}}P)\\
&\quad\quad\quad+(\beta(x)-\beta(y))\int_E d(z,y) \mu_{\bar{x}}P(dz) -\beta(x) d(x,y)\\
&\leq J_d^{x,y}(q(x,\cdot),q(y,\cdot))+\|\beta\|_\infty(1-\theta_P)d(\bar{x},\bar{y})\\
&\quad\quad\quad +(\beta(x)-\beta(y))d_\infty-\beta(x) d(x,y)
\end{align*}
This and Theorem~\ref{thm:main} entails the following corollary.
\begin{cor}
\label{cor:ZR-FV}
The coarse Ricci curvature $\sigma$ of the particle system defined by~\eqref{eq:F-for-ZR-FV} satisfies
\begin{align*}
\sigma\geq -(1-\theta_P)\|\beta\|_\infty-\sup_{x\neq y}\ \left(\frac{J_d^{x,y}(q(x,\cdot),q(y,\cdot))}{d(x,y)}+\frac{|\beta(x)-\beta(y)|}{d(x,y)}d_\infty-\beta(x)\vee \beta(y)\right)
\end{align*}
\end{cor}

\medskip

\begin{rem}
One could also consider the infinitesimal generator
\begin{align*}
Lf(x)=\int_E (f(z)-f(x))\,q(x,dz),\ \forall x\in E,
\end{align*}
and prove that
\begin{align*}
\sigma\geq -(1-\theta_P)\|\beta\|_\infty-\sup_{x\neq y}\ \left(\frac{L^c d(x,y)}{d(x,y)}+\frac{|\beta(x)-\beta(y)|}{d(x,y)}d_\infty-\beta(x)\vee \beta(y)\right)
\end{align*}
for any coupling $L^c$ of $L$.
This is also true if $L$ is not a pure jump infinitesimal generator (see an application in Example~\ref{exa:PDMP-jumps-and-killing} below).
\end{rem}

\bigskip
\begin{exa}
\label{exa:Fleming-Viot-discr} If $d$ is the {\color{black} trivial} distance $d(x,y)=\11_{x\neq y}$, then we obtain
\begin{align*}
\sigma\geq -(1-\theta_P)\|\beta\|_\infty-\sup_{x\neq y}\ \left(J_d^{x,y}(q(x,\cdot),q(y,\cdot))-\beta(x)\wedge \beta(y)\right).
\end{align*}
Note that, in several cases, $J_d^{x,y}(q(x,\cdot),q(y,\cdot){\color{black})}$ can be bounded from above using the results of Example~\ref{exa:tv-N-equal1}.
In particular, if $P_x=\delta_x$ (so that $\theta_P=0$) and $E$ is a discrete state space, one gets
\begin{align*}
\sigma\geq \inf_{x\neq y} q(x,y)+q(y,x)+\sum_{z\in E} q(x,z)\wedge q(y,z)+\beta(x)\wedge\beta(y)-\|\beta\|_\infty
\end{align*}
and hence recovers~\cite[Theorem~1.1,Remark~2.4]{CloezThai2016}.
\end{exa}

\begin{exa}
We assume that $E=\N^{\color{black}0}$, that $\beta(x)=c\11_{x=0}$ for some $c>0$ and that $q(x,\cdot)$ is the jump kernel of a birth and death process with birth and death rates respectively provided by $(b_x)_{x\in \N^{\color{black}0}}$ and $(d_x)_{x\in\N^{\color{black}0}}$ ($d_0=0$), that is
\begin{align*}
q(x,dz)=b_x\delta_{x+1}(dz)+d_x\delta_{x-1}(dz),\ \forall x\in \N^{\color{black}0}.
\end{align*}
We also assume that the process comes down from infinity, which means that $\sup_{x\in\N^{\color{black}0}} \E_x(T_0)<\infty$, where $T_0$ is the first hitting time of $0$ for the birth and death process. This is equivalent to
\begin{align*}
S:=\sum_{k\geq 1}\frac{1}{d_k\alpha_k}\sum_{l\geq k} \alpha_l <\infty,
\end{align*}
with
$\alpha_k=\left(\prod_{i=1}^{k-1} b_i\right)/\left(\prod_{i=1}^{k} d_i\right)$
 (see for instance~\cite{DoornErik1991}).

 In this case, there exist a {\color{black} bounded} function
 $\eta:\N^{\color{black}0}\rightarrow\R_+$ and a constant $\lambda_0>0$
 such that
 $b_x(\eta(x+1)-\eta(x))+d_x(\eta(x-1)-\eta(x))=-\lambda_0 \eta(x)$
 for all $x\geq 1$ ($\lambda_0$ and $\eta$ are the first eigenvalue
 and the corresponding eigenfunction for the infinitesimal generator
 of the birth and death process killed when it reaches $0$,
 see~\cite{ChampagnatVillemonais2016b} where the definition of $\eta$
 clearly implies that it is increasing {\color{black} and bounded} for birth and death
 processes {\color{black} coming down from infinity}).

Let us choose the geodesic distance $d$ on $\N^{\color{black}0}$ defined by
\begin{align*}
d(x,y)=\left|\eta(x)-\eta(y)\right|,
\end{align*}
and deduce from the computations of Example~\ref{exa:bd-proc-1d} that the coarse Ricci curvature $\sigma$ of the particle system satisfies
\begin{align*}
\sigma\geq c\left(\theta_p-\frac{\|\eta\|_\infty}{\eta(1)}\right)+\lambda_0.
\end{align*}

Consider now the Fleming-Viot type system case, \textit{i.e.} $P_x=\delta_x$. In this case, we have $\theta_P=0$, so that 
$$\sigma\geq \lambda_0-c\|\eta\|_\infty/\eta(1).$$
In particular, since this bound does not depend on $N$ and because of the convergence result of~\cite{Villemonais2014}, one can deduce that, if $c<\lambda_0\eta(1)/\|\eta\|_\infty$, then the coarse Ricci curvature is positive, uniformly in $N\geq 2$. As a consequence, a birth and death process with birth and death rates $(b_x)_{x\in\N^{\color{black}0}}$ and $(d_x)_{x\in\N^{\color{black}0}}$ and absorption rate $c\11_{x=0}$, 
converges exponentially fast toward its unique quasi-stationary distribution, conditionally on  non absorption (the details are the same as in~\cite{CloezThai2016}, where the total variation norm case is considered).
\end{exa}

\medskip
\begin{exa}
\label{exa:PDMP-jumps-and-killing}
In this example, we consider a piecewise deterministic Markov process (PDMP) evolving in $[1,+\infty)^N$ \textcolor{black}{(see~\cite{Davis1993} for a
  reference on PDMPs)}, with generator
\begin{align*}
Lf(\bar{x})=\sum_{i=1}^N\left[-x_i^2f'(x_i)\11_{x_i\geq 1}+f(x_i+1)-f(x_i)\right]+\sum_{i=1}^N\beta(x_i)\sum_{j=1}^N \int_{\R_+}\frac{f(z)-f(x_i)}{N}  P_{x_j}(dz)
\end{align*}
and the distance 
\begin{align*}
d(x,y)=\left|\exp\left(\int_1^x \frac{1}{x^2} \lambda_1(dx)\right)-\exp\left(\int_1^y \frac{1}{y^2} \lambda_1(dy)\right)\right|
\end{align*}
Each particle in this process evolves following the deterministic
dynamic $dx_t=-x_t^2dt$ and undergoes jumps of size $+1$ at rate $1$,
and jumps with respect to $\mu_{\bar{x}}P$ at rate $\beta(x_i)$.

\textcolor{black}{Consider the pure jump part of $L$, defined by
  \[\cL f(\bar{x})=\sum_{i=1}^N\beta(x_i)\sum_{j=1}^N
  \int_{\R_+}\frac{f(z)-f(x_i)}{N} P_{x_j}(dz).\] In the setting of
  Section~\ref{sec:main-result}, this corresponds to the jump measures
  $F(x,\bar{x},dz)=\delta_{x+1}+\beta(x)\mu_{\bar{x}}P$. Hence
  Theorem~\ref{thm:main} provides a coupling operator $\cL^c$ for
  $\cL$ which satisfies (following the above calculations),
  \begin{align*}
    \cL^cd(\bar{x},\bar{y})\leq -(1-\theta_P)\|\beta\|_\infty-\sup_{x\neq y}\left(\frac{|\beta(x)-\beta(y)|}{d(x,y)}d_\infty-\beta(x)\vee\beta(y)\right).
  \end{align*}
Then, considering the coupling operator $L^c$ for $L$ defined by
\begin{align*}
L^cf(\bar{x},\bar{y})=\sum_{i=1}^N\left[-x_i^2\frac{\d f}{\d x_i}(\bar{x},\bar{y})-y_i^2\frac{\d f}{\d y_i}(\bar{x},\bar{y})\right]+\cL^c f(\bar{x},\bar{y}),
\end{align*}
we deduce that}
\begin{align*}
L^c d(\bar{x},\bar{y})\leq -\left(1-(1-\theta_P)\|\beta\|_\infty-\sup_{x\neq y}\left(\frac{|\beta(x)-\beta(y)|}{d(x,y)}d_\infty-\beta(x)\vee\beta(y)\right)\right)\,d(\bar{x},\bar{y}),
\end{align*}
which entails that
\begin{align*}
\sigma\geq 1-(1-\theta_P)\|\beta\|_\infty-\sup_{x\neq y}\left(\frac{|\beta(x)-\beta(y)|}{d(x,y)}d_\infty-\beta(x)\vee\beta(y)\right).
\end{align*}
Note that, if $\beta$ is small enough and smooth enough, this provide a positive lower bound for the coarse Ricci curvature, which does not depend on $N$. In particular, applying this result to the Fleming-Viot type case and using the convergence result~\cite{Villemonais2014}, \textit{i.e.} $P_x=\delta_x$, letting $N\rightarrow+\infty$ and interpreting $\beta$ as a killing rate, one easily obtains new contraction results in $\cW_d$ for the  conditional distribution of this PDMP and also new existence/uniqueness results for the quasi-stationary distribution of this PDMP.
\end{exa}

\subsection{Birth and death processes in mean field type interaction}
\label{sec:BD-mean-field}

In~\cite{Thai2016}, the author studies, among other things, the coarse Ricci curvature of a system of particles evolving as  birth and death processes whose birth and death rates depend on the norm of the whole system, with $d(x,y)=|x-y|$. Similarly as in the cited article, we make use of the notation $d_x+q_-(x,\bar{x})$ for the death rate and $b_x+q_+(x,\bar{x})$ for the birth rate ($q_-$ and $q_+$  are allowed to depend on the position of the whole system in our case). Using the notation of Theorem~\ref{thm:main}, this means that
\begin{align*}
F(x,\bar{x},dz)=(d(x)+q_-(x,\bar{x}))\delta_{x-1}+(b(x)+q_+(x,\bar{x}))\delta_{x+1},\ \forall x\in E,\ \bar{x}\in E^N.
\end{align*}

The same calculus as in Example~\ref{exa:bd-proc-1d} (with $u_k=1$ for all $k$) shows that,
for all $x, y\in \N^{\color{black}0}$ and $\bar{x},\bar{y}\in (\N^{\color{black}0})^N$, we have, if $x< y$ and $x=y$ respectively,
\begin{align*}
J_d^{x,y}(F(x,\bar{x},\cdot),F(y,\bar{y},\cdot))&\leq 
\begin{cases}
d_x+q_-(x,\bar{x})-b_x-q_+(x,\bar{x})-d_y-q_-(y,\bar{y})+b_y+q_+(y,\bar{y}),\\
\left|q_-(x,\bar{x})-q_-(x,\bar{y})\right|+\left|q_+(x,\bar{x})-q_+(x,\bar{y})\right|.
\end{cases}
\end{align*}
Hence, if there exist some constants $a\in\R$ and $b>0$ such that
\begin{align*}
d_x+q_-(x,\bar{x})-b_x-q_+(x,\bar{x})-d_y-q_-(y,\bar{y})+b_y+q_+(y,\bar{y})\leq ad(x,y)+bd(\bar{x},\bar{y})
\end{align*}
and such that
\begin{align*}
\left|q_-(x,\bar{x})-q_-(x,\bar{y})\right|+\left|q_+(x,\bar{x})-q_+(x,\bar{y})\right|\leq bd(\bar{x},\bar{y}),
\end{align*}
then, by Theorem~\ref{thm:main}, the coarse Ricci curvature $\sigma$ of the particle system satisfies
\begin{align*}
\sigma\geq -a-b.
\end{align*}
In the particular case of the assumptions and notation of~\cite[Theorem~1.1]{Thai2016}, we can take $a=-\lambda+\alpha$ and $b=\alpha$, so that $\sigma\geq \lambda-2\alpha$ and we recover the result of the cited paper. Note that we did not need to explicitly describe a coupling  in order to obtain this bound and to slightly relax the assumptions of~\cite{Thai2016}. Also, this approach can be easily extended to other processes as in~Example~\ref{exa:bd-proc-1d-modified} for instance.

\subsection{System of particles with absolutely continuous jump measures}
\label{sec:with-density}

In this section, we assume that $E=\R^n$, $n\geq 1$, endowed with the Euclidean distance and we assume that there exist a probability measure $\zeta\in\cM_d(E)$ and two measurable functions $\alpha:E\times E^N\times E\rightarrow\R_+$ and $\beta:E\times E^N\rightarrow\R_+$ such that, for all $x\in E$ and $\bar{x}\in E^N$, $\alpha(x,\bar{x},\cdot)$ is the density of a probability measure with respect to $\zeta$ and such that
\begin{align*}
F_i(x,\bar{x},x+dz)=\beta(x,\bar{x})\alpha(x,\bar{x},z)\zeta(dz),\ \forall x\in E,\ \forall \bar{x}\in E^N,
\end{align*}
or equivalently that
\begin{align*}
F_i(x,\bar{x},dz)=\beta(x,\bar{x})\alpha(x,\bar{x},z)\zeta(dz-x),\ \forall x\in E,\ \forall \bar{x}\in E^N.
\end{align*}
For the sake of clarity, we assume that $F_i$ does not depend on $i$ (and we will set $F:=F_i$ in the rest of this subsection). However, most of the calculations considered in this section can be worked out in the general case.

The following lemma will be used together with Theorem~\ref{thm:main}
in order to compute a lower bound for the coarse Ricci curvature of
such interacting particle systems. This is particularly interesting if
one knows how to find bounds for the first moment of any probability
of type $\alpha(x,\bar{x},z)\zeta(dz)$ and for the Wasserstein
distance between any probability distributions of the same type. This
is the case for instance if the $\alpha(x,\bar{x},z)\zeta(dz)$ are
exponential laws (see Example~\ref{ex:exponential-kernel}), Gaussian
measures (see Example~\ref{ex:gaussian-kernel}) or finitely supported
discrete measures on {\color{black} $\Z$} (see Example~\ref{ex:discrete-kernel}). 

\begin{lem}
\label{le:kernel} 
Under the above settings, we have, for all $x,\bar{x},y,\bar{y}$ such
that $\beta(x,\bar{x})\geq \beta(y,\bar{y})$,
\begin{align*}
J_d^{x,y}(F(x,\bar{x},\cdot),F(y,\bar{y},\cdot))&\leq \beta(y,\bar{y})\cW_d(\alpha(x,\bar{x},z)\zeta(dz),\alpha(y,\bar{y},z)\zeta(dz))\\
&\quad\quad\quad\quad+\left(\beta(x,\bar{x})-\beta(y,\bar{y})\right)\int_E |z|\,\alpha(x,\bar{x},z)\zeta(dz).
\end{align*}
\end{lem}

\begin{proof}
For all $x,\bar{x},y,\bar{y}$ such that $\beta(x,\bar{x})\geq \beta(y,\bar{y})$, we obtain from~\eqref{eq:alt-bound},~\eqref{eq:J-prop1} and~\eqref{eq:J-prop2} that
\begin{align*}
J_d^{x,y}(F(x,\bar{x},\cdot),F(y,\bar{y},\cdot))&\leq \beta(y,\bar{y})\cW_d(\alpha(x,\bar{x},z)\zeta(dz-x),\alpha(y,\bar{y},z)\zeta(dz-y))\\
&\quad\quad +\left(\beta(x,\bar{x})-\beta(y,\bar{y})\right)\cW_d(\alpha(x,\bar{x},z)\zeta(dz-x),\delta_y)\\
&\quad\quad - \beta(x,\bar{x}) d(x,y).
\end{align*}
On the one hand, we have
\begin{multline}
\cW_d(\alpha(x,\bar{x},z)\zeta(dz-x),\alpha(y,\bar{y},z)\zeta(dz-y))\\
\leq d(x,y) + \cW_d(\alpha(x,\bar{x},z)\zeta(dz),\alpha(y,\bar{y},z)\zeta(dz))
\label{eq:crude-majoration-1}
\end{multline}
and, on the other hand, 
\begin{align}
\cW_d(\alpha(x,\bar{x},z)\zeta(dz-x),\delta_y)&\leq d(x,y)
+ \cW_d(\alpha(x,\bar{x},z)\zeta(dz-x),\delta_x)\label{eq:crude-majoration-2}\\
&= d(x,y)
+ \cW_d(\alpha(x,\bar{x},z)\zeta(dz),\delta_0)\nonumber\\
&= d(x,y)
 +\int_E |z|\,\alpha(x,\bar{x},z)\zeta(dz).\nonumber
\end{align}
This concludes the proof of Lemma~\ref{le:kernel}.
\end{proof}

\begin{rem}
  Theorem~\ref{thm:main} used in conjunction with
  Lemma~\ref{le:kernel} can only provide non-positive
  {\color{black}{lower}} bounds for the coarse Ricci curvature of the
  particle system. However, one can use such results to recover
  positive lower bounds in the case of the perturbation of a system of
  particles with known positive lower bound. {\color{black}More precisely, if an
  infinitesimal generator $L$ can be written $L=H+\varepsilon \,\cL$,
  where $H$ is known to have a positive curvature $\sigma_H>0$
  (obtained using a coupling generator $H^c$) and where one gets a
  non-positive lower bound $s_L\leq 0$ on the curvature of $\cL$
  (obtained using the above results and hence using a coupling
  operator $\cL^c$), then one deduce using the coupling operator
  $L^c=H^c+\varepsilon \cL^c$ that, for all
  $\varepsilon\in[0,-s_L/\sigma_H)$, $L$ has a positive curvature
  (this idea can typically be applied in the context of
  Remark~\ref{rem:perturbed-diff} and
  Example~\ref{exa:perturbed-diff}).}
\end{rem}

\begin{rem}
Lemma~\ref{le:kernel} is general but usually not sharp, since we used a crude upper bound in~\eqref{eq:crude-majoration-1} and~\eqref{eq:crude-majoration-2}. For instance, the case studied in Subsection~\ref{sec:BD-mean-field} enters the settings of Lemma~\ref{le:kernel}, but we obtain a better bound using a precise computation of the Wasserstein distance between measures with only three atoms. However, in the general case, the computation of the Wasserstein distance between two discrete probability measures with finite support is a difficult task.
\end{rem}

\begin{exa}
\label{ex:exponential-kernel}
In this example, we consider a process evolving in $\R$ with exponential jump measures $\alpha$ (in particular, the jumps are almost surely positive). More precisely, we assume that $\alpha(x,\bar{x},z)\zeta(dz)=\11_{z>0}\lambda(x,\bar{x})e^{-\lambda(x,\bar{x})z}\lambda(dz)$, where $\lambda(x,\bar{x})$ is a positive measurable function of $x$ and $\bar{x}$. We also assume that $\beta(x,\bar{x})$ and $\lambda(x,\bar{x})$ are anti-monotone (the larger $\beta(x,\bar{x})$, the smaller $\lambda(x,\bar{x})$).

Using~\cite{Vallender1973}, we obtain
\begin{align*}
\cW_d(\alpha(x,\bar{x},z)\zeta(dz),\alpha(y,\bar{y},z)\zeta(dz))&\leq \left|\frac{1}{\lambda(x,\bar{x})}-\frac{1}{\lambda(y,\bar{y})}\right|
\end{align*}
We also refer the reader to \cite[Examples~3.8 and~3.9]{LeyReinertEtAl2017}  for the generalization of this result to the canonical regular exponential family and to Gamma distributions respectively. 

We deduce from Lemma~\ref{le:kernel} that, if $\beta(x,\bar{x})\geq \beta(y,\bar{y})$, then
\begin{align*}
J_d^{x,y}(F(x,\bar{x}),F(y,\bar{y}))&\leq \beta(y,\bar{y})\left|\frac{1}{\lambda(x,\bar{x})}-\frac{1}{\lambda(y,\bar{y})}\right|+\frac{\beta(x,\bar{x})-\beta(y,\bar{y})}{\lambda(x,\bar{x})}\\
&= \frac{\beta(x,\bar{x})}{\lambda(x,\bar{x})}-\frac{\beta(y,\bar{y})}{\lambda(y,\bar{y})}.
\end{align*}
We deduce from Theorem~\ref{thm:main} that the coarse Ricci curvature $\sigma$ of the particle system satisfies
\begin{align*}
\sigma\geq -2\left\|\frac{\beta}{\lambda}\right\|_{Lip},
\end{align*}
where $\left\|\frac{\beta}{\lambda}\right\|_{Lip}$ is the Lipschitz norm of the function $\frac{\beta}{\lambda}:(E\times E^N,d)\rightarrow \R_+$.
\end{exa}

\begin{exa}
\label{ex:gaussian-kernel}
We consider a process evolving in $\R^n$ with Gaussian jump measures. More precisely, we assume that $\alpha(x,\bar{x},z)\sigma(dz)$ is the law of a centered Gaussian vector with co-variance matrix $\Sigma(x,\bar{x})$. For simplicity, we assume that the matrices $\Sigma(x,\bar{x})$, $x\in E,\bar{x}\in E^N$ all belong to a same commutative family of matrices.

In this case, the $W_2$-Wasserstein distance between the probability measures $\alpha(x,\bar{x},z)\sigma(dz)$ and $\alpha(y,\bar{y},z)\sigma(dz)$ is bounded {\color{black} from} above (see \cite{GivensShortt1984, KnottSmith1984,Takatsu2010,TakatsuYokota2012} and~\cite{Chafai2010} for a pedagogical account) by
\begin{align*}
\sqrt{\mathrm{Tr}(\Sigma(x,\bar{x})+\Sigma(y,\bar{y})-2(\Sigma(x,\bar{x})^{1/2}\Sigma(y,\bar{y})\Sigma(x,\bar{x})^{1/2})^{1/2})}.
\end{align*}
In particular, since the $W_2$ distance dominates the $W_d$ distance (this is an easy application of H\"older's inequality) and using the commutation of the product $\Sigma(x,\bar{x}) \Sigma(y,\bar{y})$, we deduce from Lemma~\ref{le:kernel} that, if $\beta(x,\bar{x})\geq \beta(y,\bar{y})$, then
\begin{align*}
J_d^{x,y}(F(x,\bar{x}),F(y,\bar{y}))&\leq \beta(y,\bar{y})\left\|\Sigma(x,\bar{x})^{1/2}-\Sigma(y,\bar{y})^{1/2}\right\|_{F}\\
&\quad\quad\quad
+(\beta(x,\bar{x})-\beta(y,\bar{y}))\sqrt{\sum_{i=1}^n (\Sigma(x,\bar{x})_{ii})^2}
\end{align*}
where $\|A\|_{F}=\sqrt{\sum_{i,j=1}^n (A_{ij})^2}$ is the Frobenius norm of a matrix $A\in \R^{n\times n}$.
Hence, Theorem~\ref{thm:main} entails
\begin{align*}
\sigma\geq -2\|\beta\|_\infty \left\|\Sigma^{1/2}\right\|_{Lip}-2\|\beta\|_{Lip}\left\|\sqrt{\sum_{i=1}^n (\Sigma_{ii})^2}\right\|_{\infty},
\end{align*}
where $\|\beta\|_\infty$ and $\|\beta\|_{Lip}$ are respectively the infinite norm and the Lipschitz norm of $\beta:(E\times E^N,d)\rightarrow\R_+$, $ \left\|\Sigma^{1/2}\right\|_{Lip}$ is the Lipschitz norm of $\Sigma^{1/2}:(E\times E^N,d)\rightarrow (\R^{n\times n},\|\cdot\|_{F})$ and $\left\|\sqrt{\sum_{i=1}^n (\Sigma_{ii})^2}\right\|_{\infty}$ is the infinite norm of the function
\begin{align*}
\begin{array}{ccl}
(E\times E^N,d)&\longrightarrow &\R_+\\
(x,\bar{x})&\longmapsto& \sqrt{\sum_{i=1}^n (\Sigma(x,\bar{x})_{ii})^2}.
\end{array}
\end{align*}
\end{exa}

\begin{exa}
\label{ex:discrete-kernel}
Let $E=\Z$ and assume that $\alpha(x,\bar{x},z)\sigma(dz)$ is a discrete, finitely supported probability measure. More precisely, we assume that there exists $n\geq 0$ such that
\begin{align*}
\alpha(x,\bar{x},z)\sigma(dz)=\sum_{k=-n}^n \alpha(x,\bar{x},k)\delta_k(dz),\ \forall x\in E,\,\bar{x}\in E^N.
\end{align*}

The cumulative distribution function of this measure is
\begin{align*}
F_{\alpha(x,\bar{x},z)\sigma(dz)}(t)=\sum_{k=-n}^{\lfloor t\rfloor} \alpha(x,\bar{x},k).
\end{align*}
Hence, using \cite{Vallender1973}, we obtain
\begin{align*}
\cW_d(\alpha(x,\bar{x},z)\zeta(dz),\alpha(y,\bar{y},z)\zeta(dz))&= \sum_{\ell=-n}^{n-1} \left|\sum_{k=-n}^\ell \alpha(x,\bar{x},k)-\alpha(y,\bar{y},k)\right|.
\end{align*}
We deduce from Lemma~\ref{le:kernel} that, if $\beta(x,\bar{x})\geq \beta(y,\bar{y})$, then
\begin{align*}
J_d^{x,y}(F(x,\bar{x}),F(y,\bar{y}))&\leq \beta(y,\bar{y})\sum_{\ell=-n}^{n-1} \left|\sum_{k=-n}^\ell \alpha(x,\bar{x},k)-\alpha(y,\bar{y},k)\right|.
\\
&\quad\quad\quad+(\beta(x,\bar{x})-\beta(y,\bar{y}))\sum_{k=1}^n k\,(\alpha(x,\bar{x},-k)+\alpha(x,\bar{x},k))
\end{align*}
Theorem~\ref{thm:main} implies that
\begin{align*}
\sigma\geq -2\|\beta\|_\infty \|\alpha\|_{Lip}-2\|\beta\|_{Lip}\|M_{\alpha}\|_\infty,
\end{align*}
where $M_\alpha(x,\bar{x})$ is the first absolute moment of $\alpha(x,\bar{x},\cdot)$ and where $\|\alpha\|_{Lip}$ is the Lipschitz norm of the function 
\begin{align*}
\begin{array}{ccl}
\alpha:(E\times E^N,d)&\longrightarrow &\cM(\{-n,\ldots,n\})\\
(x,\bar{x})&\longmapsto& \alpha(x,\bar{x},z)\,\zeta(dz),
\end{array}
\end{align*}
 with $\cM(\{-n,\ldots,n\})$ endowed with the norm $\|\mu\|=\sum_{k=-n}^{n} (n-k) |\mu(k)|$.
\end{exa}


\end{document}